\documentclass[11pt]{amsart}
\usepackage[utf8]{inputenc}
\usepackage{amsmath}
\usepackage{graphicx}
\graphicspath{ {./images/} }
\usepackage{svg}
\usepackage{amsfonts}
\usepackage{mathtools}
\usepackage{amssymb}
\usepackage{float}
\usepackage{ragged2e}
\usepackage{extarrows}
\usepackage{dirtytalk}
\usepackage{accents}
\usepackage{upgreek}
\usepackage{cancel}
\usepackage{hyperref}
\usepackage{url}

\DeclareMathOperator*{\esssup}{ess\,sup}
\DeclareMathOperator*{\essinf}{ess\,inf}
\usepackage{amsthm}
\usepackage[numbers]{natbib}
\usepackage{xcolor}
\newtheorem{theorem}{Theorem}[section]

\newtheorem{proposition}[theorem]{Proposition}
\newtheorem{example}[theorem]{Example}
\newtheorem{lemma}[theorem]{Lemma}
\newtheorem{definition}[theorem]{Definition}
\newtheorem{remark}[theorem]{Remark}

\numberwithin{equation}{section}
\usepackage{pgfplots}
\pgfplotsset{compat=1.18}
\definecolor{mygreen1}{RGB}{20, 180, 120}
\usepackage{soul}
\usepackage{siunitx}

\title[Symmetrization and the Poisson equation]{Symmetrization of measures and the one-dimensional Poisson equation with Dirichlet boundary conditions}
\author{Christos Papadimitriou}
\address{Aristotle University of Thessaloniki\\
         Department of Mathematics\\
         Thessaloniki, Greece, 546 35}
\email{papadimitc@math.auth.gr}
\date{June 2025}

\begin{document}
\begin{abstract}
Let \(\mu\) be a finite Borel measure on \((-\pi,\pi)\). Consider the one-dimensional Poisson equation \(-u''=\mu\), where equality holds in the sense of distributions, with Dirichlet boundary conditions \(u(\pm\pi)=0\). In this paper, we define measures that are transformations of \(\mu\), we compare the convex integral means of the original solutions \(u_\mu\) and the transformed ones, and we prove the uniqueness of a solution that maximizes the convex integral means. 
\end{abstract}
\maketitle
\section{Introduction}
Imagine a metal rod of length \(2\pi\) that is heated through by an external heat source with distribution \(f\in L^1[-\pi,\pi]\). We are interested in the temperature \(u_f\) of the steady-state of the rod, meaning that the temperature of the rod does not change over time, if we suppose that the end-points of the rod are kept frozen.
To put this into equations, we have that \(u_f\) has to satisfy the Poisson equation
\begin{equation}\label{poissoneq}
    -u''=f\text{ on \([-\pi,\pi]\)}
\end{equation}
with Dirichlet boundary conditions
\begin{equation}\label{dirichletcond}
    u(-\pi)=u(\pi)=0.
\end{equation}
Another way to view this thought experiment is to let the rod interact with its environment considering the end-points of the rod to be submerged in a cold bath of temperature zero and allow heat be dissipated via Newton's law of cooling, i.e. the heat flux at the end-points is proportional to the temperature. In this case, the steady-state of the rod has to satisfy Robin boundary conditions, i.e. for some constant \(\alpha>0\) (which depends on the material of the rod or the type of liquid) we have
\begin{equation}\label{Robin}
-u'(-\pi)+\alpha u(-\pi)=u'(\pi)+\alpha u(\pi)=0.
\end{equation}
With this setup, we can approximate Dirichlet boundary conditions by letting \(\alpha\) tend to \(+\infty\). The physical interpretation of this is that the liquid that surrounds the rod, not only starts with temperature zero, but also remains frozen and keeps the end-points of the rod frozen as well, thus it has infinite thermal capacity.
\par Recently, J. J. Langford and P. McDonald \cite{Langford}, studied the problems (\ref{poissoneq}), (\ref{Robin}) and (\ref{poissoneq}), (\ref{dirichletcond}) and were interested in comparing different solutions and the changes of specific properties of the solutions \(u_f\), after performing rearrangement transformations to the heat source \(f\). Namely, they were interested in the convex integral means of the solutions and, as a consequence, their \(L^p\)-norms.
\par Approaching the Dirichlet problem through the Robin problem has the benefit that enables us to apply known results about the Robin problem (for example, theorems from \cite{Langford} or \cite{egw}) in our context, by simply using the Theorem of Dominated Convergence. A more interesting question would be to investigate what happens if the heat source is not a Lebesgue integrable function but rather a more general object, a Borel measure.
\par Throughout this paper, we will use the term \textit{Borel measure on \((-\pi,\pi)\)} when referring to a positive measure on the Borel \(\sigma\)-algebra of \((-\pi,\pi)\) that is not the zero measure. Let \(\mu\) be a finite Borel measure on \((-\pi,\pi)\). We will show (see Proposition \ref{Representation}) that there exists a unique absolutely continuous function \(u_\mu\) on \([-\pi,\pi]\) that solves the Poisson equation
\begin{equation}\label{Poisson}
    -u''=\mu \text{ on \((-\pi,\pi)\) in the sense of distributions}
\end{equation}
and satisfies Dirichlet boundary conditions
\begin{equation}\label{Dirichlet}
    u(-\pi)=u(\pi)=0.
\end{equation}
We will talk more about distributions later and give equivalent expressions of \((\ref{Poisson})\).\\
\noindent
\textbf{Notation.}
\begin{itemize}
    \item Let \(\lambda\) denote the one-dimensional Lebesgue measure. As usual, if a measure \(\mu\) is absolutely continuous with respect to the Lebesgue measure, we write \(\mu\ll\lambda\), and if a measure \(\sigma\) is singular to the Lebesgue measure, we write \(\sigma\perp\lambda\).
    \item In the special case that \(d\mu=fd\lambda\) for some \(f\in L^1[-\pi,\pi]\), we will denote the solution \(u_\mu\) to the problem (\ref{Poisson}), (\ref{Dirichlet}) by \(u_f\).
\end{itemize}
\vspace{2mm}
For the general definition of {\it symmetric decreasing rearrangement} (s.d.r. for short) of functions, we refer to \cite{Baernstein}. In Section 3 we present a special case for the s.d.r. of functions that better suits our context. Moreover, we will introduce the symmetrization of finite Borel measures. More will be said about this in Section 3, but, let us give a brief insight here.
\par Let \(\mu\) be a finite Borel measure on \((-\pi,\pi)\), let \(\sigma\) be the singular part of \(\mu\) and \(f\in L^1[-\pi,\pi]\) the Radon--Nikodym derivative of \(\mu\). We define the finite Borel measure \(\mu^\#\), and call it the symmetrization of \(\mu\), to be the measure such that the Radon--Nikodym derivative of \(\mu^\#\) is the s.d.r. \(f^\#\) of \(f\) and the singular part of \(\mu^\#\) is \(\sigma((-\pi,\pi))\delta_0\), which is a constant multiple of the Dirac measure \(\delta_0\) with a unit point mass at 0. To condense all this, if
\begin{equation}
d\mu=fd\lambda+d\sigma,
\end{equation}
then we define \(\mu^\#\) such that
\begin{equation}
d\mu^\#=f^\#d\lambda+\sigma((-\pi,\pi))d\delta_0.
\end{equation}
For more details, see Definitions in Subsection 3.2.
\par Our main result is the following.
\begin{theorem}[S.d.r.\ and convex integral means]\label{sdrconvex}
Let \(\mu\) be a finite Borel measure on \((-\pi,\pi)\) and let \(\mu^\#\) be its symmetrization. Let \(u_\mu\) and \(u_{\mu^\#}\) be the corresponding solutions to the Dirichlet problem \((\ref{Poisson})\), \((\ref{Dirichlet})\). Then
\begin{equation}\label{ineq sdr}
\int_{-\pi}^\pi \phi\bigl(u_\mu(x)\bigr)dx\leq\int_{-\pi}^\pi \phi\bigl(u_{\mu^\#}(x)\bigr)dx
\end{equation}
for any convex and increasing function \(\phi:\mathbb{R}\to\mathbb{R}.\)
\par Moreover, if \(\phi\) is strictly convex, then \((\ref{ineq sdr})\) holds as an equality if and only if the Radon--Nikodym derivative \(f\) of \(\mu\) is equal to its s.d.r. \(f^\#\) a.e. on \((-\pi,\pi)\) and the singular part of \(\mu\) is a Dirac mass at the origin.
\end{theorem}
In particular, for \(p\geq1\), consider the convex and increasing function
\[
\phi(s)=
\begin{cases}
0, & s\leq0,\\
s^p, &s\geq0.
\end{cases}
\]
Then, \((\ref{ineq sdr})\) gives inequalities for the \(L^p\)-norms of \(u_\mu\) and \(u_{\mu^\#}\). Letting \(p\to+\infty\) we get the inequality for their \(L^\infty\)-norms, i.e.
\begin{equation}\label{inequality of Lp norms eisawgh}
||u_\mu||_{L^p}\leq||u_{\mu^\#}||_{L^p} \text{, for }1\leq p\leq+\infty.
\end{equation}
In Section 6, we deal with the case of equality in \((\ref{inequality of Lp norms eisawgh})\).
\begin{remark}
\textnormal{ In the case when \(\mu\) is an absolutely continuous measure, inequality \((\ref{ineq sdr})\) was proved in \cite{Langford}. In that paper, the case of equality was not investigated.}
\end{remark}
The paper is structured as follows. In Section 2, we present some fundamental facts about distributions as well as the general properties of the solutions \(u_\mu.\) Section 3 is dedicated to finite Borel measures. In that section, we talk about the Lebesgue--Radon--Nikodym decomposition of a measure, we introduce the notion of polarizations and symmetrizations for measures, and finally we give our second main theorem (Theorem \ref{polarconvex}), which is a comparison theorem for polarization. Section 4 contains preliminary results that will be used to prove the main theorems; these proofs appear in Section 5. Finally, Section 6 connects the symmetrization of measures and the \(L^p-\)norms of solutions.
\section{Preliminaries}
\subsection{Distributions on \texorpdfstring{\((-\pi,\pi)\)}{TEXT}} In this subsection we present the fundamentals of distributions. The general cases for everything we present here without a proof can be found in \cite[Chapter 9]{Folland}. We shall start off with some notation.
\par Let \(C_c^\infty(-\pi,\pi)\) denote the set of infinitely differentiable real functions with compact support in \((-\pi,\pi)\). Now, consider the following notion of convergence in \(C_c^\infty(-\pi,\pi)\): Let \(\{\psi_n\}_{n\in\mathbb{N}}\subset C_c^\infty(-\pi,\pi)\). We say that the sequence \(\{\psi_n\}\) converges to zero in \(C_c^\infty(-\pi,\pi)\) if and only if:\\
1) there exists a compact set \(K\subset(-\pi,\pi)\) such that \(\operatorname{supp}\psi_n\subseteq K\) 
for all \(n\in\mathbb{N}\),\\
2) for each nonnegative integer \(a\), we have
\[\max_{x\in K}\left|\frac{\partial^{a}}{\partial x^{a}}\psi_n(x)\right|\xrightarrow{n\to+\infty}0.\]
We say that \(\{\psi_n\}\subset C_c^\infty(-\pi,\pi)\) converges to \(\psi\) in \(C_c^\infty(-\pi,\pi)\) if and only if \(\{\psi_n-\psi\}\) converges to zero in \(C_c^\infty(-\pi,\pi)\). The set of functions \(C_c^\infty(-\pi,\pi)\) equipped with this notion of convergence will be symbolized with \(\mathcal{D}(-\pi,\pi)\) and we call it the space of test functions in \((-\pi,\pi)\).
\par A continuous linear functional \(T\) on the space of test functions \(\mathcal{D}(-\pi,\pi)\) is called a distribution on \((-\pi,\pi).\) We shall use the notation \(<T,\psi>\) to denote the value of \(T\) acting on the test function \(\psi\). The set of all distributions on \((-\pi,\pi)\) is symbolized with \(\mathcal{D}'(-\pi,\pi)\). We call two distributions \(T_1\), \(T_2\) on \(\mathcal{D}'(-\pi,\pi)\) equal if and only if \(<T_1,\psi>=<T_2,\psi>\) for all \(\psi\in\mathcal{D}(-\pi,\pi).\)
\begin{example} \textnormal{Let \(f\in L^1[-\pi,\pi]\). Then \[<T_f,\psi>\coloneqq\int_{-\pi}^\pi f(x)\psi(x)dx,\text{ for }\psi\in\mathcal{D}(-\pi,\pi),\]
defines a distribution on \((-\pi,\pi)\). Analogously, if \(\mu\) is a Borel measure on \((-\pi,\pi)\), then
\[<T_\mu,\psi>=\int_{-\pi}^\pi\psi(x) d\mu(x)\text{ for }\psi\in\mathcal{D}(-\pi,\pi),\]
also defines a distribution on \((-\pi,\pi)\). It is very common to identify \(T_f\) with \(f\) and \(T_\mu\) with \(\mu\).}
\end{example}
\begin{definition}[Distributional derivatives]
\textnormal{Let \(T\in\mathcal{D}'(-\pi,\pi)\). The distribution \(T'\) such that \(<T',\psi>=-<T,\psi'>\) for all \(\psi\in\mathcal{D}(-\pi,\pi)\) is called the distributional derivative of \(T\).}
\end{definition}
This process can be repeated ad infinitum. In this paper, we will use only the first and second distributional derivatives. The second distributional derivative \(T''\) of \(T\in\mathcal{D}'(-\pi,\pi)\) is given by
\[<T'',\psi>=-<T',\psi'>=-(-<T,\psi''>)=<T,\psi''>.\]
\begin{example}\textnormal{Let \(x_0\in (-\pi,\pi)\) be a fixed point. The Dirac measure \(\delta_{x_0}\) with point mass at \(x_0\) defines a distributions in the following way:
\[<\delta_{x_0},\psi>=\psi(x_0),\text{ for }\psi\in\mathcal{D}(-\pi,\pi).\]}
\end{example}
\subsection{The Dirichlet problem for the distributional Poisson equation}
\begin{definition}\textnormal{
The Green's function for the Dirichlet boundary problem on the interval \([-\pi,\pi]\)
is 
\begin{equation}\label{Green}
G(x,y)=-\frac{1}{2\pi}xy-\frac{1}{2}|x-y|+\frac{\pi}{2},\quad x,y\in[-\pi,\pi].
\end{equation}}
\end{definition}
It is trivial to see that \(G\) vanishes on the boundary of the square \([-\pi,\pi]\times[-\pi,\pi]\), is strictly positive in the interior of the square, and attains global maximum \(G(0,0)=\pi/2.\) Also, \(G(x,y)=G(y,x)\) for all \(x,y\in[-\pi,\pi].\)
\begin{proposition}\label{G distonos}
For \(x,y\in(-\pi,\pi)\), the second distributional derivative with respect to \(x\) of \(G(x,y)\) is \(-\delta_y(\{x\})\).
\begin{proof}
Fix \(y_0\in(-\pi,\pi).\) Then
\[G_x(x,y_0)=\frac{\partial G}{\partial x}(x,y_0)=
\begin{dcases}
-\frac{1}{2\pi}y_0+\frac{1}{2}, & x\in(-\pi,y_0),\\
-\frac{1}{2\pi}y_0-\frac{1}{2}, & x\in(y_0,\pi).
\end{dcases}\]
So, for the second distributional derivative with respect to \(x\) of \(G(x,y_0)\), we have
\begin{align*}
<G_{xx},\psi>&=-<G_x,\psi'>=-\int_{-\pi}^\pi G_x(x,y_0)\psi'(x)dx\\
&=-\int_{-\pi}^{y_0}\bigl(-\frac{1}{2\pi}y_0+\frac{1}{2}\bigr)\psi'(x)dx-\int^{\pi}_{y_0}\bigl(-\frac{1}{2\pi}y_0-\frac{1}{2}\bigr)\psi'(x)dx\\
&=\bigl(\frac{1}{2\pi}y_0-\frac{1}{2}\bigr)\int_{-\pi}^{y_0} \psi'(x)dx+\bigl(\frac{1}{2\pi}y_0+\frac{1}{2}\bigr)\int_{y_0}^\pi\psi'(x)dx\\
&=\bigl(\frac{1}{2\pi}y_0-\frac{1}{2}\bigr)\bigl(\psi(y_0)-\psi(-\pi)\bigr)+\bigl(\frac{1}{2\pi}y_0+\frac{1}{2}\bigr)\bigl(\psi(\pi)-\psi(y_0)\bigr)\\
&=-\psi(y_0).
\end{align*}
But \(\displaystyle<\delta_{y_0},\psi>=\int_{-\pi}^\pi \psi(x)d\delta_{y_0}(x)=\psi(y_0).\)
\end{proof}
\end{proposition}
\begin{figure}[htp]\centering
\begin{tabular}{@{}cc@{}}
\begin{tikzpicture}[define rgb/.code={\definecolor{mycolor}{RGB}{#1}},
                    rgb color/.style={define rgb={#1},mycolor}]
\begin{axis}
[  
    x=8mm,
    y=8mm,
    xtick={ 0 },
    xmin=-3.5,
    xmax=3.5,
    xlabel={\tiny $y$},
    extra x ticks = { -pi , -3 , -2 , -pi/2 , -1 , 0 , pi/4 , 1 , 2 , 3 , pi },
    extra x tick labels = { $-\pi$ , $\phantom{.}$,  $\phantom{.}$ , $-\frac{\pi}{2}$ , $\phantom{.}$, $0$ , $\frac{\pi}{4}$ , $\phantom{.}$ , $\phantom{.}$ , $\phantom{.}$, $\pi$ },
    axis x line=middle,
    ytick={2},
    tick label style={font=\tiny},
    extra y ticks = {1},
    extra y tick labels = {$\phantom{.}$},
    ymin=0,
    ymax=2.5,
    axis y line=middle,
    no markers,
    samples=100,
    domain=-pi:pi,
    restrict y to domain=-pi:pi,
]
\addplot[blue] coordinates { (-pi,0) (-pi/2,3*pi/8) (pi,0) }; 
\addplot[black] coordinates { (-pi,0) (0,pi/2) (pi,0) };
\addplot[red] coordinates { (-pi,0) (pi/4,15*pi/32) (pi,0) };
\end{axis}
\end{tikzpicture}
&\begin{tikzpicture}[define rgb/.code={\definecolor{mycolor}{RGB}{#1}},
                    rgb color/.style={define rgb={#1},mycolor}]
\begin{axis}
[  
    x=8mm,
    y=8mm,
    xtick={0},   
    xmin=-4,
    xmax=4,
    xlabel={\tiny $x$},
    extra x ticks = {-pi , pi/2 , pi},
    extra x tick labels = {$-\pi$,  $y_0$ , $\pi$},
    axis x line=middle,
    ytick={-2,-1,0,1,2},
    tick label style={font=\tiny},
    ymin=-1.5,
    ymax=1.5,
    axis y line=middle,
    no markers,
    samples=100,
    domain=-pi:pi,
    restrict y to domain=-pi:pi
]
\addplot[black] coordinates { (-pi,1/4) (pi/2,1/4) };
\addplot[black] coordinates { (pi/2,-3/4) (pi,-3/4) };
\addplot[densely dotted] coordinates { (-pi,1) (pi,0) };
\addplot[densely dotted] coordinates { (-pi,0) (pi,-1) };
\filldraw[color=black, fill=white] (pi/2,1/4) circle (2pt);
\filldraw[color=black, fill=white] (pi/2,-3/4) circle (2pt);
\filldraw[black] (-pi,1) circle (0.5pt) node[anchor=south]{\tiny{$(-\pi,1)$}};
\filldraw[black] (pi,-1) circle (0.5pt) node[anchor=north]{\tiny{$(\pi,-1)$}};
\filldraw[black] (-pi,1/4) circle (0.5pt);
\filldraw[black] (pi,-3/4) circle (0.5pt);
\end{axis}
\end{tikzpicture}
\\ (a) & (b) \end{tabular}
\caption{(a) Graphs of \(\color{blue}{G(-\pi/2,y)}\), \(\color{black}{G(0,y)}\) and \(\color{red}{G(\pi/4,y)}\); (b) Graph of \(G_x(x,y_0)\).}
\end{figure}
The following lemma is a slight improvement of the classic ``differentiation under the integral sign" theorem. It will be used to prove the absolute continuity of solutions \(u_\mu\) in Proposition \ref{Representation}. We omit the proof of the lemma, since it can be achieved by following the proof of \cite[Theorem 2.27]{Folland} almost to the letter, and then applying \cite[Theorem 3.35]{Folland}.
\begin{lemma}\label{Lemma}
Let \((X,\mathcal{M},\mu)\) be a measure space. Suppose that a function \(f:X\times[a,b]\to\mathbb{C}\), where \(-\infty<a<b<+\infty\), satisfies the following conditions:
\begin{enumerate}
    \item[1.] \(f(x,t)\) is a measurable function of \(x\) and \(t\) jointly, and \(f(\cdot,t)\in L^1(\mu)\) for all \(t\in[a,b]\) held fixed.
    \item[2.] For all \(x\in X\), \(f(x,t)\) is an absolutely continuous function of \(t\).
    \item[3.] \(\partial f/ \partial t\) belongs to \(L^1(a,b)\).
\end{enumerate}
Then \(F(t)=\int_X f(x,t)d\mu(x)\) is an absolutely continuous function of \(t\), and for almost every \(t\in [a,b]\), its derivative exists and is given by
\[F'(t)=\frac{d}{dt}\int_X f(x,t)d\mu(x)=\int_X\frac{\partial}{\partial t}f(x,t)d\mu(x).\]
\end{lemma}
\begin{proposition}\label{Representation}
Let \(\mu\) be a finite Borel measure on \((-\pi,\pi)\). There exists a unique absolutely continuous function \(u_\mu\) on \([-\pi,\pi]\) that solves the Poisson equation
\[-u_\mu''=\mu\text{ on } (-\pi,\pi) \text{ in the sense of distributions}\]
and satisfies the Dirichlet boundary conditions \(u_\mu(-\pi)=u_\mu(\pi)=0\). The solution \(u_\mu\) can be represented as
\begin{equation}\label{u_mu}
u_\mu(x)=\int_{-\pi}^\pi G(x,y)d\mu(y).
\end{equation}
\end{proposition}
\begin{proof}
First, we prove the uniqueness. Let \(u\) and \(v\) be absolutely continuous solutions to the problem \((\ref{Poisson}),(\ref{Dirichlet}).\) Consider \(w(x)=u(x)-v(x)\), of which the second distributional derivative is the zero distribution. Subsequently, the first distributional derivative of \(w\) is a constant distribution, which means that there exists \(a\in\mathbb{R}\) such that \(w'(x)=a\) a.e. on \([-\pi,\pi]\) and thus there also exists \(b\in\mathbb{R}\) such that \(w(x)=ax+b\) a.e. on \([-\pi,\pi]\). We assumed that \(u\) and \(v\) are continuous, and so is \(w\), therefore \(w(x)=ax+b\) on \([-\pi,\pi]\). Also, \(w(\pm\pi)=u(\pm\pi)-v(\pm\pi)=0\), so \(a=b=0\) and \(w\equiv0\).
\par For the existence, we shall prove that \(u_\mu\), as given in \((\ref{u_mu})\), solves the problem. Trivially, \(u_\mu\) satisfies the boundary conditions \(u_\mu(\pm\pi)=0\), since \(G(x,\pm\pi)=0.\) Also, it is easy to verify that \(u_\mu\) satisfies the assumptions of Lemma \ref{Lemma} and thus it is absolutely continuous on \([-\pi,\pi]\). To see that \(-u_\mu''=\mu\) in the sense of distributions, we apply Fubini's Theorem and Proposition \ref{G distonos}:
\begin{align*}
-<u_\mu'',\psi>&=-<u_\mu,\psi''>=-\int_{-\pi}^\pi u_\mu(x)\psi''(x)dx\\
&=-\int_{-\pi}^\pi\int_{-\pi}^\pi G(x,y)d\mu(y)\psi''(x)dx\\
&=-\int_{-\pi}^\pi\int_{-\pi}^\pi G(x,y)\psi''(x)dxd\mu(y)\\
&=-\int_{-\pi}^\pi<G,\psi''>d\mu(y)=-\int_{-\pi}^\pi<G_{xx},\psi>d\mu(y)\\
&=-\int_{-\pi}^\pi(-\psi(y))d\mu(y)=\int_{-\pi}^\pi\psi(y)d\mu(y)=<\mu,\psi>.
\end{align*}
\end{proof}
\begin{remark}\label{remark thetikes luseis}
\textnormal{By Proposition \(\ref{Representation}\), for any finite Borel measure \(\mu\) on \((-\pi,\pi)\), the solution \(u_\mu\) to the Dirichlet problem is strictly positive on \((-\pi,\pi)\).}
\end{remark}
\subsection{Examples} We give some examples of solutions to the Dirichlet problem \((\ref{Poisson}),\) \((\ref{Dirichlet})\) for various measures \(\mu\). Note that these measures will have the same total variation. This occurs because the total variation is invariant under the measure transformations that we introduce in the next section.
\begin{example}[Lebesgue measure]
\[u_\lambda(x)=\pi^2-\frac{1}{4}\bigl((x+\pi)|x+\pi|+(\pi-x)|\pi-x|\bigr).\]
\end{example}
\begin{example}[Dirac measure with mass at zero]\label{example dirac}
\[u_{2\pi\delta_0}(x)=\pi^2-\pi|x|=2\pi G(x,0).\]
\end{example}
More generally, \(u_{a\delta_{x_0}}(x)=aG(x,x_0),\) where \(a>0\) and \(x_0\in(-\pi,\pi).\) 
\begin{example}[Sequence of absolutely continuous measures that tents to Dirac]\textnormal{
Consider the measures \(\mu_n\) such that \(d\mu_n=2\pi\dfrac{n}{2}\mathcal{\chi}_{[-\frac{1}{n},\frac{1}{n}]}d\lambda\). Then
\[u_{\mu_n}(x)=
\begin{dcases}
\pi^2+\pi x, & x\in\bigl[-\pi,-\frac{1}{n}\bigr],\\
-\frac{\pi n x^2}{2}+\pi^2-\frac{\pi}{2n}, & x\in\bigl[-\frac{1}{n},\frac{1}{n}\bigr],\\
\pi^2-\pi x, & x\in\bigl[\frac{1}{n},\pi\bigr].
\end{dcases}\]
Note that \(u_{\mu_n}\equiv u_{2\pi\delta_0}\) on \(\displaystyle[-\pi,\pi]\setminus\bigl(-\frac{1}{n},\frac{1}{n}\bigr)\).}
\end{example}
\begin{figure}[htp]\centering
\begin{tabular}{@{}cc@{}}
\begin{tikzpicture}[define rgb/.code={\definecolor{mycolor}{RGB}{#1}},
                    rgb color/.style={define rgb={#1},mycolor}]
\begin{axis}
[  
    x=8mm,
    y=8mm,
    xtick={-pi,-1,-0.5,0.5,1,pi},
    xmin=-3.5,
    xmax=3.5,
    xlabel={\tiny $x$},
    axis x line=middle,
    ytick={pi^2/2,pi^2-pi/2,pi^2-pi/4,pi^2},
    tick label style={font=\tiny},
    ymin=0,
    ymax=10,
    axis y line=middle,
    no markers,
    samples=200,
    domain=-pi:pi,
    restrict y to domain=-pi:pi,
]
\addplot[mygreen1] {pi*pi-((x+pi)*abs(x+pi)+(pi-x)*abs(pi-x))/4}; 
\addplot[black] {pi^2-pi*abs(x)};
\addplot[domain=-1:1][blue] {-pi*x^2/2+pi^2-pi/2};
\addplot[domain=-1/2:1/2][red] {-pi*x^2+pi^2-pi/4};
\draw[dotted] (-1,0) -- (-1,pi^2-pi);
\draw[dotted] (1,0) -- (1,pi^2-pi);
\draw[dotted] (-0.5,0) -- (-0.5,pi^2-pi/2);
\draw[dotted] (0.5,0) -- (0.5,pi^2-pi/2);
\end{axis}
\end{tikzpicture}
&\begin{tikzpicture}[define rgb/.code={\definecolor{mycolor}{RGB}{#1}},
                    rgb color/.style={define rgb={#1},mycolor}]

\end{tikzpicture}
\end{tabular}
\caption{Graphs of \(\color{mygreen1}{u_\lambda(x)}\), \(\color{black}{u_{2\pi\delta_0}(x)}\), \(\color{blue}{u_{\mu_1}}(x)\) and \(\color{red}{u_{\mu_2}}(x)\).}
\end{figure}
\section{Finite Borel measures}
\subsection{Lebesgue--Radon--Nikodym decomposition}
\begin{definition}\label{decompdef}\textnormal{ Let \(\mu\) be a finite Borel measure on \((-\pi,\pi).\)
\begin{itemize}
    \item If \(\mu(\{x\})=0\) for all \(x\in(-\pi,\pi)\), then \(\mu\) is called continuous.
    \item If there exists a (possibly finite) sequence \(\{x_i\}\) of distinct points in \((-\pi,\pi)\) and a sequence of numbers \(\{a_i\}\in \ell^1\) such that \(\mu=\sum_i a_i \delta_{x_i},\) where \(\delta_{x_i}\) is the Dirac measure with unit mass at the point \(x_i\), then \(\mu\) is called discrete or purely discontinuous.
\end{itemize}}
\end{definition}
Note that all discrete measures are singular with respect to the Lebesgue measure and all absolutely continuous measures with respect to the Lebesgue measure are continuous. The converse of neither of the above statements is true.\\
\vspace{0mm}
\par The following Theorem can be found in the literature for measures that are regular Borel (\cite[Theorem 19.61]{Hewitt}).
\begin{theorem}\label{decomp}
Let \(\mu\) be a (positive) regular Borel measure on \(\mathbb{R}\). Then \(\mu\) can be expressed in a unique way as
\[\mu=\nu+\sigma+\delta,\]
where \(\nu,\sigma\) and \(\delta\) are (positive) regular Borel measures on \(\mathbb{R}\), \(\nu\ll\lambda\), \(\sigma\perp\lambda\), \(\delta\) is purely discontinuous and \(\sigma\) is continuous.
\end{theorem}
\begin{remark}
\textnormal{According to \cite[Theorem A.2.2]{Ransford_1995}, if \(\mu\) is a finite positive Borel measure on a metric space \(X\), then \(\mu\) is a regular measure. Since \((-\pi,\pi)\) is a metric space with the euclidean metric, the measures throughout our paper are regular, and thus Theorem \(\ref{decomp}\) applies to them. Moreover, considering the absolutely continuous part \(\nu\) of a positive Borel measure \(\mu\), note that the Radon--Nikodym derivative \(f\) of \(\nu\) is Lebesgue integrable, since \(\nu\) is finite, and \(f\geq0\) a.e. on \([-\pi,\pi]\), since \(\nu\) is positive.}
\end{remark}
\subsection{Symmetrization and polarizations of measures}
The polarizations and the symmetric decreasing rearrangement of functions are well established notions in literature, and the reader can find the more general definitions of them in books such as \cite{Baernstein}. In our paper, since we work on the interval \([-\pi,\pi]\) and have Dirichlet boundary conditions, we can assume that the functions and measures with which we work are identically zero outside \([-\pi,\pi]\) (more details on the thought process that led to the definitions below can be found on \cite{egw}). With the assumptions described above, we can now present the definitions of s.d.r. and polarization in a way that better fits our context, but also respects the general definitions. Furthermore, we introduce the notion of symmetrization of Borel measures on \((-\pi,\pi)\) and we generalize the polarization of functions to the notion of polarization of measures that do not have a singular continuous part.
\begin{definition}[Symmetric decreasing rearrangement of functions and symmetrization of measures]\label{sdr}\quad
\begin{itemize}
    \item \textnormal{Let \(f\in L^1[-\pi,\pi]\).\ We define \(f^\ast:[0,2\pi]\to\overline{\mathbb{R}}\coloneqq\mathbb{R}\cup\{-\infty,+\infty\}\) as 
\[f^\ast(t)=
\begin{dcases}
\esssup f, & t=0,\\
\inf\bigl\{s\in\mathbb{R}:\lambda\bigl(\{x\in [-\pi,\pi]:f(x)>s\}\bigr)\leq t\bigr\}, &t\in(0,2\pi),\\
\essinf f, &t=2\pi.
\end{dcases}\]
The function \(f^\ast\) is called the decreasing rearrangement of \(f\).}
\par \textnormal{We also define \(f^\#:[-\pi,\pi]\to\overline{\mathbb{R}}\) the symmetric decreasing rearrangement of \(f\), given by \(f^\#(t)=f^\ast(2|t|)\) for \(t\in[-\pi,\pi]\).
\item Let \(\mu\) be a finite Borel measure on \((-\pi,\pi)\). Let \(\mu=\nu+\sigma+\delta\) be the decomposition of \(\mu\) as described in Theorem \(\ref{decomp}\). Let \(f\in L^1[-\pi,\pi]\) such that \(d\nu=fd\lambda\) and let \(M=\sigma\bigl((-\pi,\pi)\bigr)+\delta\bigl((-\pi,\pi)\bigr)\). We define the symmetrization of \(\mu\) to be the measure \(\mu^\#\), such that
\[d\mu^\#=f^\#d\lambda+Md\delta_0.\]}
\end{itemize}
\end{definition}
\begin{definition}[Polarization of functions]\label{defpolarfun} \textnormal{Let \(b\in(-\pi,0)\cup(0,\pi)\) and \(H\) be the open half-line of \(\mathbb{R}\setminus\{b\}\) which contains \(0\). Let \(f\in L^1[-\pi,\pi]\).
\begin{itemize}
    \item If \(b\in(0,\pi),\) we define the polarization (towards zero) of \(f\) with respect to \(b\) (or the polarization of \(f\) with respect to \(H=(-\infty,b)\)) as
    \[f_H(x)=
    \begin{dcases}
        f(x), &x\in[-\pi,2b-\pi),\\
        \max\{f(x),f(2b-x)\}, &x\in[2b-\pi,b],\\
        \min\{f(x),f(2b-x)\}, &x\in[b,\pi].
    \end{dcases}\]
    \item If \(b\in(-\pi,0),\) we define the polarization (towards zero) of \(f\) with respect to \(b\) (or the polarization of \(f\) with respect to \(H=(b,+\infty)\)) as
    \[f_H(x)=
    \begin{dcases}
        \min\{f(x),f(2b-x)\}, &x\in[-\pi,b],\\
        \max\{f(x),f(2b-x)\}, &x\in[b,2b+\pi],\\
        f(x), &x\in(2b+\pi,\pi].
    \end{dcases}\]
\end{itemize}}
\end{definition}
In a completely analogous way and under the previous assumptions for \(b\) and \(H\), one can define the polarization of a purely discontinuous measure \(\delta=\sum_i a_i\delta_{x_i}\) with respect to \(b\) to be the purely discontinuous measure \(\delta_H\) such that
\[\delta_H(\{x\})=\begin{dcases}
    \delta(\{x\}), &x\in(-\pi,2b-\pi),\\
    \max\{\delta(\{x\}),\delta(\{2b-x\})\}, &x\in[2b-\pi,b],\\
    \min\{\delta(\{x\}),\delta(\{2b-x\})\}, &x\in[b,\pi),
\end{dcases}\]
for \(b\in(0,\pi)\), and
\[\delta_H(\{x\})=\begin{dcases}
    \min\{\delta(\{x\}),\delta(\{2b-x\})\}, &x\in(-\pi,b],\\
    \max\{\delta(\{x\}),\delta(\{2b-x\})\}, &x\in[b,2b+\pi],\\
    \delta(\{x\}), &x\in(2b+\pi,\pi),
\end{dcases}\]
for \(b\in(-\pi,0)\).\\
As stated above, \(\delta_H\) is a purely discontinuous measure on \((-\pi,\pi)\) and it can be written as \(\delta_H=\sum_i a_i\delta_{\tilde{x}_i}\), where \(\tilde{x}_i=x_i\) in the case of \(\delta_H(\{x_i\})=\delta(\{x_i\})\), or \(\tilde{x}_i=2b-x_i\) if \(2b-x_i\in(-\pi,\pi)\) and \(\delta_H(\{2b-x_i\})=\delta(\{x_i\})\).\\
This leads to the following definition.
\begin{definition}[Polarization of a measure without singular continuous part]\label{polar measure}\textnormal{
Let \(\mu\) be a finite Borel measure on \((-\pi,\pi)\) that can be written as \(\mu=\nu+\delta\), where \(\delta\) is purely discontinuous and \(\nu\ll\lambda\) with \(d\nu=fd\lambda\) for some \(0\leq f\in L^1[-\pi,\pi]\). Let \(b\in(-\pi,0)\cup(0,\pi)\) . We define the polarization of \(\mu\) with respect to \(b\) to be the measure \(\mu_H=\nu_H+\delta_H\), where \(\nu_H\ll\lambda\) such that \(d\nu_H=f_H d\lambda\), and \(\delta_H\) is the purely discontinuous measure as described above.}
\end{definition}
\begin{example}\label{paradeigma}\textnormal{
Fix a point \(x_0\in[-\pi,\pi]\) and consider the Green's function \(G(x_0,\cdot):[-\pi,\pi]\to[0,+\infty)\) with one argument being constant. It is trivial to see that \(G(0,\cdot)\) is its own s.d.r and thus no polarization towards zero affects it. Assume that \(x_0\neq0\) and let \(b\in(-\pi,0)\cup(0,\pi)\). Let \(\bigl(G(x_0,\cdot)\bigr)_H\) be the polarization of \(G(x_0,\cdot)\) with respect to \(b\). If \(x_0<0\) (respectively \(x_0>0\)), then \(\bigl(G(x_0,\cdot)\bigr)_H\equiv G(x_0,\cdot)\) if and only if \(b\in(-\pi,x_0]\cup(0,\pi)\) (respectively \(b\in(-\pi,0)\cup[x_0,\pi)\)).}
\end{example}
We can now state the second main theorem of this paper. For its proof, see Subsection 4.2.
\begin{theorem}[Polarization and convex integral means]\label{polarconvex}
Let \(\mu\) be a finite Borel measure on \((-\pi,\pi)\) with Lebesgue--Radon--Nikodym decomposition that does not include a singular continuous part. Let \(b\in(-\pi,0)\cup(0,\pi)\). Let \(\mu_H\) be the polarization of \(\mu\) with respect to \(b\). Let \(u_\mu\) and \(u_{\mu_H}\) denote the corresponding solutions to the Dirichlet problem \((\ref{Poisson})\), \((\ref{Dirichlet})\). Then
\begin{equation}\label{ineq polar}
\int_{-\pi}^\pi \phi\bigl(u_\mu(x)\bigr)dx\leq\int_{-\pi}^\pi \phi\bigl(u_{\mu_H}(x)\bigr)dx
\end{equation}
for any convex and increasing function \(\phi:\mathbb{R}\to\mathbb{R}.\)
\par Furthermore, if \(\phi\) is strictly increasing, then \((\ref{ineq polar})\) holds as an equality if and only if \(\mu=\mu_H\) on \((-\pi,\pi)\).
\end{theorem}
\subsection{Hardy-Littlewood inequalities}
The next two propositions are special cases of the Hardy-Littlewood inequalities for the polarization and the symmetric decreasing rearrangement, respectively. The reader can find them in their most general form as Theorems 2.9 and 2.15 of \(\cite{Baernstein}\).
\begin{proposition}\label{hardy polar}
Let \(0\leq f,g\in L^1[-\pi,\pi]\). Let \(b\in(-\pi,0)\cup(0,\pi)\) and \(H\) be the open half-line of \(\mathbb{R}\) that contains 0. Let \(f_H\) and \(g_H\) denote the polarizations of \(f\) and \(g\) with respect to \(b\). Then
\[\int_{-\pi}^\pi f(y)g(y)dy\leq\int_{-\pi}^\pi f_H(y)g_H(y)dy.\]
Equality holds if and only if the set \(A_H=B_H\cup C_H\) has zero Lebesgue measure, where
\[B_H=\{y\in H\cap[-\pi,\pi]:f(y)<f(2b-y)\text{ and }g(y)>g(2b-y)\}\]
and
\[C_H=\{y\in H\cap[-\pi,\pi]:f(y)>f(2b-y)\text{ and }g(y)<g(2b-y)\}.\]
\end{proposition}
\begin{proposition}\label{hardy sdr}
Let \(0\leq f,g\in L^1[-\pi,\pi]\), and let \(f^\#\) and \(g^\#\) be their symmetric decreasing rearrangements. Then
\[\int_{-\pi}^\pi f(y)g(y)dy\leq\int_{-\pi}^\pi f^\#(y)g^\#(y)dy.\]
Equality holds if and only if the set
\[A=\bigl\{(x,y)\in[-\pi,\pi]^2:f(x)<f(y)\text{ and }g(x)>g(y)\bigr\}\]
has zero two-dimensional Lebesgue measure.
\end{proposition}
We can now prove a preliminary symmetrization result.
\begin{theorem}\label{theorem max}
Let \(\mu\) be a finite Borel measure on \((-\pi,\pi)\) and let \(\mu^\#\) be its symmetrization. Let \(u_\mu\) and \(u_{\mu^\#}\) be the solutions to the corresponding Dirichlet problems. Then
\begin{equation}\label{max}u_{\mu^\#}(0)=\max_{[-\pi,\pi]} u_{\mu^\#}\geq \max _{[-\pi,\pi]}u_\mu.\end{equation}
Equality holds if and only if \(\mu=\mu^\#\).
\end{theorem}
Of course, \((\ref{max})\) could be seen as a consequence of Theorem \ref{sdrconvex} (see \((\ref{inequality of Lp norms eisawgh})\)). In Section 5, however, we will use \((\ref{max})\) to prove Theorem \ref{sdrconvex}.
\begin{proof}
All solutions to the Dirichlet problem vanish on \(\pm\pi\), and so we have nothing of value to show for the ends of the interval. Fix \(x_0\in(-\pi,\pi)\). Consider the decomposition \(\mu=\nu+\sigma+\delta\), where \(d\nu=fd\lambda\) for some \(0\leq f\in L^1[-\pi,\pi]\), \(\sigma\) is a continuous measure on \((-\pi,\pi)\) that is singular with respect to \(\lambda\) and \(\delta=\sum_i a_i \delta_{x_i}\) is a purely discontinuous measure, as described in Definition \ref{decompdef}. Recall that \(\mu\) is not the zero measure; so at least one of the measures
\(\nu\), \(\sigma\), \(\delta\) is not zero. Then
\begin{align*}
u_{\mu}(x_0)&=u_f(x_0)+u_\sigma(x_0)+u_\delta(x_0)\\
&=\int_{-\pi}^\pi G(x_0,y)f(y)dy+\int_{-\pi}^\pi G(x,y)d\sigma(y)+\sum_{i}a_i G(x_0,x_i).
\end{align*}
Note that 
\begin{equation}\label{sdr ineq green} 
G^\#(x_0,y)\leq G(0,y) \text{ for all }y\in[-\pi,\pi]
\end{equation}
and equality holds for any \(y\in(-\pi,\pi)\) if and only if \(x_0=0\). To verify this, we use the definition of s.d.r. to see that \(G(0,y)=(\pi-|y|)/2=G^\#(0,y)\) and \(\displaystyle G^\#(x_0,y)=\frac{G(x_0,x_0)}{\pi/2}G(0,y)\).
\par First, we will prove that \(u_f(x_0)\leq u_{f^\#}(0)\) for all \(x_0\in(-\pi,\pi)\), with the equality holding if and only if \(f=f^\#\) a.e. on \((-\pi,\pi)\). Let \(b\in(-\pi,0)\cup(0,\pi)\) and let \(f_H\) be the polarization of \(f\) with respect to \(b\). Let \(f^\#\) be the s.d.r of \(f\). Applying Propositions \ref{hardy polar} and \ref{hardy sdr} and using \((\ref{sdr ineq green})\), we get the following inequalities
\begin{align}\displaystyle
\label{efarmogh hardy polar}    u_f(x_0)&\leq\int_{-\pi}^\pi G_H(x_0,y)f_H(y)dy\\
\label{efarmogh hardy sdr}    &\leq\int_{-\pi}^\pi (G_H)^\#(x_0,y)(f_H)^\#(y)dy=\int_{-\pi}^\pi G^\#(x_0,y)f^\#(y)\\
\label{max Green austhro}    &\leq\int_{-\pi}^\pi G(0,y)f^\#(y)dy=u_{f^\#}(0),
\end{align}
So, 
\begin{equation}\label{max of u f}
u_f(x_0)\leq u_{f^\#}(0)=\max u_{f^\#}\text{ for all }x_0\in(-\pi,\pi).
\end{equation}
\par To show the uniqueness statement of the theorem for the absolutely continuous part of \(\mu\), we work as follows. Assume that \(\nu\) is not zero, then \(f\), \(f_H\) and \(f^\#\) are not a.e. equal to zero.  If \(x_0\neq0\), then \((\ref{max Green austhro})\) holds as a strict inequality, since \(f^\#\) cannot be identically zero a.e. on \((-\pi,\pi)\) and \(G^\#(x_0,y)\leq G(0,y)\) for all \(y\in[-\pi,\pi]\), where the equality holds if and only if \(x_0=0\). So, for \(x_0\neq0\), \((\ref{max of u f})\) holds as a strict inequality. Let \(x_0=0\) and assume that \(f\) is not a.e. equal to \(f^\#\). Then, by \cite[Lemma 1.38]{Baernstein}, there exists some \(b\in(-\pi,0)\cup(0,\pi)\) such that \(f\) is not a.e. equal to \(f_H\), where \(f_H\) is the polarization of \(f\) with respect to \(b\). Let us compute the set \(A_H\) of Proposition \ref{hardy polar} for \(g(y)=G(0,y)\). It is almost trivial to see that \(G(0,y)>G(0,2b-y)\) for all \(y\in H\), where \(H\) is as defined in Definition \ref{defpolarfun}. Thus
\begin{align*}
B_H&=\{y\in H\cap[-\pi,\pi]:f(y)<f(2b-y)\text{ and }G(0,y)>G(0,2b-y)\}\\
&=\{y\in H\cap[-\pi,\pi]:f(y)<f(2b-y)\}
\end{align*}
and
\[C_H=\{y\in H\cap[-\pi,\pi]:f(y)>f(2b-y)\text{ and }G(0,y)<G(0,2b-y)\}
=\varnothing.\]
So \(A_H=\{y\in H:f(y)<f(2b-y)\}\).
The condition that \(\lambda(A_H)=0\) turns out to be equivalent to \(f(y)\geq f(2b-y)\) almost everywhere on \(H\) which is equivalent to \(f=f_H\) a.e. on \(H.\) But, as mentioned above, \(f\) is not a.e. equal to \(f_H\), thus \(\lambda(A_H)>0\) and, by Proposition \ref{hardy polar}, \((\ref{efarmogh hardy polar})\), and subsequently \((\ref{max of u f})\), holds as a strict inequality.
\par Next, we will show that \(u_\delta(x_0)\leq u_{\delta^\#}(0)\) for all \(x_0\in(-\pi,\pi)\), with the equality holding if and only if \(\delta=\delta^\#\) on \((-\pi,\pi)\). Since \(G(x,y)\) attains maximum only for \(x=y=0\), we have
\begin{align}\label{max of u d}
\nonumber u_\delta(x_0)&=\sum_i a_i G(x_0,x_i)\\
&\leq\sum_i a_i G(0,0)
=G(0,0)\sum_i a_i=u_{M\delta_0}(0),
\end{align}
where \(M=\sum_i a_i=\delta((-\pi,\pi))\). By Definition \ref{sdr}, we have \(M\delta_0=\delta^\#\), thus
\begin{equation}\label{max of u delta}
u_\delta(x_0)\leq u_{\delta^\#}(0)=\max u_{\delta^\#}\text{ for all }x_0\in(-\pi,\pi).
\end{equation}
\par Assume that \(\delta\) is not zero. If \(x_0\neq0\) or any \(x_i\) is different from the point zero, then \((\ref{max of u d})\) holds as a strict inequality, since \(G(0,0)\) is the unique global maximum of \(G\). Therefore, \((\ref{max of u delta})\) cannot hold as an equality if \(\delta\neq\delta^\#\).
\par Finally, for the singular continuous part \(\sigma\) of \(\mu\), we will use yet another approach to show that \(u_\sigma(x_0)\leq u_{\sigma^\#}(0)\) for all \(x_0\in(-\pi,\pi)\), and in fact we will see that if \(\sigma\) is not the zero measure, then the inequality is sharp. Since \(G(x,y)\) attains global maximum \(G(0,0)\), we have 
\begin{align*}
u_\sigma(x_0)&=\int_{-\pi}^\pi G(x_0,y)d\sigma(y)\leq\int_{-\pi}^\pi G(0,0)d\sigma(y)\\
&=G(0,0)\sigma\bigl((-\pi,\pi)\bigr)=G(0,0)\sigma^\#\bigl((-\pi,\pi)\bigr)=u_{\sigma^\#}(0).
\end{align*}
Therefore
\begin{equation}\label{max of u sigma}
u_\sigma(x_0)\leq u_{\sigma^\#}(0)=\max u_{\sigma^\#}\text{ for all }x_0\in(-\pi,\pi).
\end{equation}
Combining inequalities \((\ref{max of u sigma})\), \((\ref{max of u f})\) and \((\ref{max of u delta})\), we get the inequality \((\ref{max})\).
\par  Assume that \(\sigma\) is not zero. We will show that inequality \((\ref{max of u sigma})\) is strict, which implies that \((\ref{max})\) is also strict. Note that \(\sigma^\#\) is the only measure on \((-\pi,\pi)\) such that \[\sigma^\#\bigl((-\varepsilon,\varepsilon)\bigr)=\sigma^\#\bigl((-\pi,\pi)\bigr)=\sigma\bigl((-\pi,\pi)\bigr)\text{ for all }0<\varepsilon\leq\pi.\]
Note also that \(\sigma\neq\sigma^\#\), since \(\sigma\) is a continuous measure and \(\sigma^\#\) is a Dirac mass at zero, and thus \(\mu\neq\mu^\#\). Moreover, there exists \(\varepsilon_1>0\) such that \(\sigma\bigl((-\pi,\pi)\setminus(-\varepsilon_1,\varepsilon_1)\bigr)>0\) and this inequality also holds for every \(\varepsilon\in(0,\varepsilon_1)\).\\
If \(x_0\neq0\), then for every \(0<\varepsilon<\min\{\varepsilon_1,|x_0|\}\) we have \(G(x_0,y)\leq G(x_0,x_0)\) for every \(y\in(-\pi,\pi)\setminus(-\varepsilon,\varepsilon)\) and, thus
\begin{align}
\nonumber u_\sigma(x_0)&=\int_{(-\varepsilon,\varepsilon)} \kern-0.5cm G(x_0,y)d\sigma(y)+\int_{(-\pi,\pi)\setminus(-\varepsilon,\varepsilon)}\kern-1.6cmG(x_0,y)d\sigma(y)\\
\nonumber&\leq\int_{(-\varepsilon,\varepsilon)} \kern-0.5cm G(0,0)d\sigma(y)+\int_{(-\pi,\pi)\setminus(-\varepsilon,\varepsilon)}\kern-1.6cmG(x_0,x_0)d\sigma(y)\\
\nonumber&=G(0,0)\sigma\bigl((-\varepsilon,\varepsilon)\bigr)+G(x_0,x_0)\sigma\bigl((-\pi,\pi)\setminus(-\varepsilon,\varepsilon)\bigr)\\
\label{max u sigma}&<G(0,0)\sigma\bigl((-\varepsilon,\varepsilon)\bigr)+G(0,0)\sigma\bigl((-\pi,\pi)\setminus(-\varepsilon,\varepsilon)\bigr)=u_{\sigma^\#}(0).
\end{align}
Inequality \((\ref{max u sigma})\) is strict because \(x_0\neq0.\)\\
If \(x_0=0\), we have \(G(0,y)\leq G(0,\varepsilon_1)\) for every \(y\in(-\pi,\pi)\setminus(-\varepsilon_1,\varepsilon_1)\) and thus
\begin{align}
\nonumber u_\sigma(0)&=\int_{(-\varepsilon_1,\varepsilon_1)} \kern-0.7cm G(0,y)d\sigma(y)+\int_{(-\pi,\pi)\setminus(-\varepsilon_1,\varepsilon_1)}\kern-1.8cmG(0,y)d\sigma(y)\\
\nonumber&\leq\int_{(-\varepsilon_1,\varepsilon_1)} \kern-0.7cm G(0,0)d\sigma(y)+\int_{(-\pi,\pi)\setminus(-\varepsilon_1,\varepsilon_1)}\kern-1.8cmG(0,\varepsilon_1)d\sigma(y)\\
\nonumber&=G(0,0)\sigma\bigl((-\varepsilon_1,\varepsilon_1)\bigr)+G(0,\varepsilon_1)\sigma\bigl((-\pi,\pi)\setminus(-\varepsilon_1,\varepsilon_1)\bigr)\\
\label{u sigma mhden}&<G(0,0)\sigma\bigl((-\varepsilon_1,\varepsilon_1)\bigr)+G(0,0)\sigma\bigl((-\pi,\pi)\setminus(-\varepsilon_1,\varepsilon_1)\bigr)=u_{\sigma^\#}(0).
\end{align}
Inequality \((\ref{u sigma mhden})\) is strict because \(\varepsilon_1>0\), and this concludes our proof.
\end{proof}
\begin{remark}
\textnormal{By Remark \(\ref{remark thetikes luseis}\), \(\displaystyle\min_{[-\pi,\pi]} u_\mu=u_\mu(\pm\pi)=0\), so the temperature gap or oscillation of \(u_\mu\) is
\begin{equation}
\operatorname{osc}(u_\mu)=\max_{[-\pi,\pi]}u_\mu-\min_{[-\pi,\pi]}u_\mu=\max_{[-\pi,\pi]}u_\mu.
\end{equation}
Thus, Theorem \textnormal{\ref{theorem max}} automatically also solves a maximal problem about the oscillation of solutions to the Dirichlet problem.
The oscillation of solutions to the Robin problem was studied by D. Betsakos and A. Solynin in \textnormal{\cite{Betsakos}}, and, as already mentioned in the introduction, the Robin problem is closely linked to the Dirichlet problem.}
\end{remark}
\section{Karamata type inequalities}
The following lemma and its proof are based on Lemma 1 and Theorem 1 of \cite{Karamata}.
\begin{lemma}[Karamata's inequality]\label{karamata}
Let \(\phi:[0,M]\to\mathbf{\mathbb{R}}\) be a convex and increasing function, where \(M>0\). Let \(x_1,x_2\) and \(y_1,y_2\) be two pairs of real numbers contained in \([0,M]\) such that
\begin{enumerate}
    \item[(a)] \(x_1\geq x_2\) and \(y_1\geq y_2\),
    \item[(b)] \(x_1\leq y_1\),
    \item[(c)] \(x_1+x_2\leq y_1+y_2\).
\end{enumerate}
Under these assumptions, the following inequality holds
\begin{equation}\label{karamata ineq}
\phi(x_1)+\phi(x_2)\leq\phi(y_1)+\phi(y_2).
\end{equation}
\end{lemma}
\begin{proof}
If \(x_1=y_1\), then \(x_2\leq y_2\) and thus \((\ref{karamata ineq})\) holds trivially, by the monotonicity of \(\phi.\) Therefore, we may assume that \(y_2<x_2\leq x_1<y_2\). Because \(\phi\) is increasing, we have
\[c_2\coloneqq\frac{\phi(y_2)-\phi(x_2)}{y_2-x_2}\geq0,\]
and because \(\phi\) is convex, we have
\[c_1\coloneqq\frac{\phi(y_1)-\phi(x_1)}{y_1-x_1}\geq\frac{\phi(y_2)-\phi(x_2)}{y_2-x_2}=c_2.\]
Thus,
\begin{align*}
\sum_{i=1}^2\phi(y_i)-\sum_{i=1}^2\phi(x_i)&=\phi(y_1)-\phi(x_1)+\phi(y_2)-\phi(x_2)\\
&=c_1(y_1-x_1)+c_2(y_2-x_2)\\
&=c_1(y_1-x_1)+c_2(y_2+y_1-y_1-x_2-x_1+x_1)\\
&=c_1(y_1-x_1)-c_2(y_1-x_1)+c_2(y_2+y_1-x_2-x_1)\\
&=(c_1-c_2)(y_1-x_1)+c_2(y_1+y_2-x_1-x_2)\geq0.
\end{align*}
\end{proof}
\begin{lemma}\label{karamata green}
Let \(G\) be the Green's function for the Dirichlet problem on \([-\pi,\pi]\), let \(b\in(-\pi,0)\cup(0,\pi)\). If \(b>0\) (respectively \(b<0\)), let \(I\) be the interval \([b,\pi]\) (respectively \([-\pi,b]\)). For all \(x,y\in I\), the following two inequalities hold
\begin{equation}\label{4.2}
G(x,y)+G(x',y)\leq G(x,y')+G(x',y')
\end{equation}
and
\begin{equation}\label{4.3}
G(x,y)+G(x,y')\leq G(x',y)+G(x',y'),
\end{equation}
where \(x'\) and \(y'\) are the reflections of \(x\) and \(y\) with respect to \(b\), i.e. \(x'=2b-x\) and \(y'=2b-y\).\\
\end{lemma}
\begin{proof}
Putting down the formulas for \(G(x,y)\), \(G(x',y)\), \(G(x,y')\) and \(G(x',y')\), we see that the inequality \((\ref{4.2})\) is equivalent to 
\begin{equation}\label{4.4}
\frac{xy}{\pi}+|x-y|+\frac{x'y}{\pi}+|x'-y|\geq\frac{xy'}{\pi}+|x-y'|+\frac{x'y'}{\pi}+|x'-y'|.
\end{equation}
But, \(x'=2b-x\) and \(y'=2b-y\), so \(|x'-y'|=|x-y|\) and \(|x-y'|=|x'-y|\). Thus \((\ref{4.4})\) is equivalent to 
\begin{equation}\label{4.5}
    by\geq b^2,
\end{equation}
which, for \(b>0\), holds for all \(y\geq b\), i.e. for all \(y\in I=[b,\pi]\), and similarly, for  \(b<0\), holds for all \(y\in I=[-\pi,b]\).\\
Inequality \((\ref{4.3})\) can be proved in the same way.
\end{proof}
\begin{lemma}\label{lemma karamata u f}
Let \(0\leq f\in L^1[-\pi,\pi]\). Let \(b\in(-\pi,0)\cup(0,\pi),\) let \(I\) be as defined in Lemma \textnormal{\ref{karamata green}} and let \(f_H\) be the polarization of \(f\) with respect to \(b\). The following hold:
\begin{enumerate}
    \item[(a)] If \(x\in I\), let \(x'=2b-x\). Then
    \begin{equation}\label{uf syn uf}
    u_f(x)+u_f(x')\leq u_{f_H}(x)+u_{f_H}(x').
    \end{equation}
    \item[(b)] If \(x\in I\), let \(x'=2b-x\). Then
    \begin{equation}\label{uf uf polwsh polwsh}
    u_f(x)\leq u_{f_H}(x').
    \end{equation}
    \item[(c)] If \(b>0\) (respectively \(b<0\)), let \(x\in[-\pi,b]\) (respectively \(x\in [b,\pi]\)). Then 
    \begin{equation}\label{uf uf polwsh}
    u_f(x)\leq u_{f_H}(x).
    \end{equation}
    For \(b>0\) (respectively \(b<0\)), the following are equivalent:
    \begin{enumerate}
        \item[(i)] \((\ref{uf uf polwsh})\) holds as an equality for some \(x\in(-\pi,b]\) (resp. \(x\in [b,\pi)\)).
        \item[(ii)] \((\ref{uf uf polwsh})\) holds as an equality for all \(x\in(-\pi,b]\) (resp. \(x\in [b,\pi)\)).
        \item[(iii)] \(u_f\equiv u_{f_H}\) on \([-\pi,\pi]\).
        \item[(iv)] \(f=f_H\) almost everywhere on \([-\pi,\pi]\).
    \end{enumerate}
\end{enumerate}
\end{lemma}
\begin{proof}
Due to symmetry, it suffices to prove the lemma for \(b\in(-\pi,0)\).\\
(a) Let \(x\in[-\pi,b]\) and \(x'=2b-x\in[b,2b+\pi].\) From \((\ref{4.2})\) we have that \(G(x,\cdot)+G(x',\cdot)\) coincides with its polarization with respect to \(b\). Thus, by Proposition \ref{hardy polar},
    \begin{align*}
    u_f(x)+u_f(x')&=\int_{-\pi}^\pi\bigl(G(x,y)+G(x',y)\bigr)f(y)dy\\
    &\leq\int_{-\pi}^\pi\bigl(G(x,y)+G(x',y)\bigr)_H f_H(y)dy\\
    &=\int_{-\pi}^\pi\bigl(G(x,y)+G(x',y)\bigr)f_H(y)dy=
    u_{f_H}(x)+u_{f_H}(x').
    \end{align*}
 (b) Let \(x\in[-\pi,b]\) and consider \(x'=2b-x\in[b,2b+\pi].\) By the definition of polarization, \(f\equiv f_H\) a.e. on \((2b+\pi,\pi]\). 
It is easy to see that \(G(x,y)\leq G(x',y)\) for all \(y\in[2b+\pi,\pi]\), and thus\\
    \[\int_{2b+\pi}^\pi G(x,y)f(y)dy\leq\int_{2b+\pi}^\pi G(x',y)f(y)dy=\int_{2b+\pi}^\pi G(x',y)f_H(y)dy.\]
By the definition of polarization, we have \(f_H(y)=\min\{f(y),f(y')\}\) for all \(y\in[-\pi,b],\) where \(y'=2b-y\). Also, \(f_H(y')=\max\{f(y),f(y')\}\), since \(y'\in[b,2b-y]\), and thus \(f_H(y)\leq f_H(y')\) for all \(y\in[-\pi,b].\)\\
Note that \(\{f_H(y),f_H(y')\}=\{f(y),f(y')\}\), thus we can distinguish two possible cases:\\
either 
    \[f_H(y)=f(y)\text{ and } f_H(y')=f(y')\]
or
\[f_H(y)=f(y')\text{ and } f_H(y')=f(y).\]
Assume that the former case is true, then \(f(y)=f_H(y)\leq f_H(y')=f(y')\).
By Lemma \ref{karamata green}, we have \(G(x,y)-G(x',y)\leq G(x',y')-G(x,y'),\) and, thus
\[G(x,y)f(y)-G(x',y)f_H(y)\leq G(x',y')f_H(y')-G(x,y')f(y').\]
Similarly, if the latter case holds, then \(f(y')=f_H(y)\leq f_H(y')=f(y)\). Inequality (\ref{4.3}) implies that \(G(x,y')-G(x',y)\leq G(x',y')-G(x,y)\). Thus,
\[G(x,y')f(y')-G(x',y)f_H(y)\leq G(x',y')f_H(y')-G(x,y)f(y).\]
We conclude that in both cases, for \(y\in[-\pi,b]\) we have
\[G(x,y)f(y)+G(x,y')f(y')\leq G(x',y)f_H(y)+G(x',y')f_H(y'),\]
and, integrating over \([-\pi,b]\) with respect to \(y\), we get
\[\int_{-\pi}^b\bigl(G(x,y)f(y)+G(x,y')f(y')\bigr)dy\leq\!\!\int_{-\pi}^b\bigl(G(x',y)f_H(y)+G(x',y')f_H(y')\bigr)dy,\]
and with a change of variables this implies that
\[\int_{-\pi}^{2b+\pi} G(x,y)f(y)dy\leq\int_{-\pi}^{2b+\pi} G(x',y)f_H(y)dy.\]
We have already shown that
\[\int_{2b+\pi}^\pi G(x,y)f(y)dy\leq\int_{2b+\pi}^\pi G(x',y)f_H(y)dy,\]
and, adding these two inequalities, we end up with \((\ref{uf uf polwsh polwsh})\).\\
(c) Let \(x\in[b,\pi]\). By Example \ref{paradeigma}, \(G(x,\cdot)\equiv \bigl(G(x,\cdot)\bigr)_H\). Thus, by Proposition \ref{hardy polar}, we have
\begin{align*}
u_f(x)&=\int_{-\pi}^\pi G(x,y)f(y)dy\leq\int_{-\pi}^\pi \bigl(G(x,y)\bigr)_H f_H(y)dy\\
&=\int_{-\pi}^\pi G(x,y)f_H(y)dy=u_{f_H}(x).
\end{align*}
Considering the equivalence statement on Lemma \ref{lemma karamata u f}(c), it is trivial to see that \((iv)\implies(iii)\implies(ii)\implies(i)\), and to close the circle, we just have to show that \((i)\implies(iv)\).\\
Assume that \((\ref{uf uf polwsh})\) holds as an equality for some \(x\in[b,\pi).\)
By Proposition \ref{hardy polar}, we have that \(\lambda(B_H\cup C_H)=0\), where
\[B_H=\{y\in(b,\pi]:f(y)<f(2b-y)\text{ and }G(x,y)>G(x,2b-y)\}\]
and
\[C_H=\{y\in(b,\pi]:f(y)>f(2b-y)\text{ and }G(x,y)<G(x,2b-y)\}.\]
By Example \(\ref{paradeigma}\), \(G(x,y)> G(x,2b-y)\) for all \(y\in(b,2b+\pi)\). Note that we have considered  \(G\) to  be identically zero outside of \([-\pi,\pi]^2\). Thus, \(C_H=\varnothing\) and \(B_H=\{y\in(b,2b+\pi):f(y)<f(2b-y)\}\). So, \(f=f_H\) a.e. on \((b,2b+\pi)\) and, by the definition of polarization, \(f=f_H\) a.e. on \([-\pi,\pi]\).
\end{proof}
\begin{lemma}\label{karamata u d}
Let \(\delta\) be a purely discontinuous measure on \((-\pi,\pi)\), as defined in Definition \(\ref{decompdef}\). Let \(b\in(-\pi,0)\cup(0,\pi)\), let \(\mu_H\) be the polarization of \(\mu\) with respect to \(b\), as given in Definition \(\ref{defpolarfun}\). Let \(I\) be as defined in Lemma \textnormal{\ref{karamata green}}. The following hold:
\begin{enumerate}
    \item[(a)] If \(x\in I\), let \(x'=2b-x\). Then
    \begin{equation}\label{u d syn u d}
    u_\delta(x)+u_\delta(x')\leq u_{\delta_H}(x)+u_{\delta_H}(x').
    \end{equation}
    \item[(b)] If \(x\in I\), let \(x'=2b-x\). Then
    \begin{equation}\label{u d u d polwsh polwsh}
    u_\delta(x)\leq u_{\delta_H}(x').
    \end{equation}
    \item[(c)] If \(b>0\) (respectively \(b<0\)), let \(x\in[-\pi,b]\) (respectively \(x\in [b,\pi]\)). Then 
    \begin{equation}\label{u d u d polwsh}
    u_\delta(x)\leq u_{\delta_H}(x).
    \end{equation}
    For \(b>0\) (respectively \(b<0\)), the following are equivalent
    \begin{enumerate}
        \item[(i)] \((\ref{u d u d polwsh})\) holds as an equality for some \(x\in(-\pi,b]\) (resp. \(x\in [b,\pi)\)).
        \item[(ii)] \((\ref{u d u d polwsh})\) holds as an equality for all \(x\in(-\pi,b]\) (resp. \(x\in [b,\pi)\)).
        \item[(iii)] \(u_\delta\equiv u_{\delta_H}\) on \([-\pi,\pi]\).
        \item[(iv)] \(\delta=\delta_H\).
    \end{enumerate}
\end{enumerate}
\end{lemma}
\begin{proof}
Let us denote by \(K\) the set \(\operatorname{supp}\delta=\{x_i\}\). We shall prove the lemma for \(b\in(-\pi,0).\) The proof for the other case is completely analogous.\\
(a) Let \(x\in[-\pi,b]\) and \(x'=2b-x\in[b,2b+\pi].\) We want to show that
\begin{equation}\label{karamata u delta}
u_\delta(x)+u_\delta(x')\leq u_{\delta_H}(x)+u_{\delta_H}(x') \text{ for all }x\in[-\pi,b],
\end{equation}
i.e.
\[\sum_i a_i G(x,x_i)+\sum_i a_i G(x',x_i)\leq\sum_i a_iG(x,\tilde{x}_i)+\sum_i a_iG(x',\tilde{x}_i),\]
where either \(\tilde{x}_i=x_i\) or \(\tilde{x}_i=2b-x_i=x'_i.\) If \(\tilde{x}_i=x_i\), then obviously \(a_iG(x,x_i)+a_iG(x',x_i)=a_iG(x,\tilde{x}_i)+a_iG(x',\tilde{x}_i)\), and if \(\tilde{x}_i=x'_i\), then, by Lemma \(\ref{karamata green}\), \(a_iG(x,x_i)+a_iG(x',x_i)\leq a_iG(x,\tilde{x}_i)+a_iG(x',\tilde{x}_i)\). With summation over all indices \(i\), we get the inequality \((\ref{karamata u delta})\).\\
For (b) and (c), note that \(u_{a_1\mu_1+a_2\mu_2}=a_1u_{\mu_1}+a_2u_{\mu_2}\). Consider three simple cases:
\begin{enumerate}
    \item[(A)] The measure \(\delta\) is a single unit mass \(\delta_{x_i}\) such that \(b\leq x_i<\pi\) and \(x_i'\notin K\).
    \item[(B)] The measure \(\delta\) is a single unit mass \(\delta_{x_i}\) such that \(-\pi<x_i<b\) and \(x_i'\notin K\).
    \item[(C)] The measure \(\delta\) is the sum of unit masses at two points which are symmetric with respect to \(b\), i.e. \(\delta_{x_i}+\delta_{x'_i}\), where \(-\pi<x_i<b<x'_i<2b+\pi\) and \(x_i'=2b-x_i\).
\end{enumerate}
We assumed that, if \(\delta_{x_i}\) belongs to case \((A)\), then \(x_i'\notin K\) (and the analogous for case \((B)\)), to avoid the possibility that case \((C)\) can further decompose into the sum of two measures, one from case \((A)\) and one from \((B)\).\\
Note that the polarization with respect to \(b\) does not affect cases \((A)\) and \((C)\) but it does affect the measures of case \((B)\)\\
We can construct our original measure \(\delta\) with linear combinations of cases \((A)\), \((B)\) and \((C)\), where \(\) thus it suffices to prove the rest of Lemma \(\ref{karamata u d}\) just for these three cases.\\
(b) Let \(x\in[-\pi,b]\) and \(x'=2b-x\in[b,2b+\pi].\)\\
In case \((A)\), \((\ref{u d u d polwsh polwsh})\) is equivalent to \(G(x,x_i)\leq G(x,\tilde{x}_i)\) for all \(x\in[-\pi,b]\). Since \(b\leq x_i\), we have \(\tilde{x}_i=x_i\) and, by Example \ref{paradeigma}, \(G(x,x_i)\leq G(x',x_i)\) for all \([-\pi,b].\)\\
In case \((B)\), since \(x_i<b\), we have \(\tilde{x}_i=x_i'=2b-x_i\) and thus \((\ref{u d u d polwsh polwsh})\) is equivalent to \(G(x,x_i)\leq G(x',x_i')\) for all \(x\in[-\pi,b]\). It is trivial to see that this last inequality is equivalent to \(x+x_i\leq 2b\), which holds for all \(x\in[-\pi,b]\) since \(x_i<b\).\\
In case \((C)\), we want to show that \(G(x,x_i)+G(x,x_i')\leq G(x',x_i)+G(x',x_i')\) for all \(x\in I\). Since \(x_i\in I\), we have that this inequality is true, by Lemma \(\ref{karamata green}\).\\
(c) Let \(x\in[b,\pi]\). In cases \((A)\) and \((C)\), we have that \(\delta=\delta_H\), and thus \((\ref{u d u d polwsh})\) holds trivially as an equality.\\
In case \((B)\), inequality \((\ref{u d u d polwsh})\) is equivalent to \(G(x,x_i)\leq G(x,2b-x_i)\) for all \(x\in[b,\pi]\). This last inequality is quite elementary and in fact it is strict for all \(x\in[b,\pi)\), since \(x_i\neq x_i'\).\\
Finally, we deal with the equivalence statement of Lemma \ref{karamata u d}(c). Trivially, we have that \((iv)\implies(iii)\implies(ii)\implies(i)\) and we just have to show that \((i)\implies(iv)\) . To achieve this, we will use contraposition. Assume that \(\delta_H\) is not equal to \(\delta\). Then \(\delta\) is equal to a linear combination of measures of cases \((A)\), \((B)\) and \((C)\) and at least one of the components is a positive scalar product of a measure of case \((B)\). We saw that
\[u_{\delta_{x_i}}(x)<u_{\delta_{x_i'}}(x)\text{ for all }x\in[b,\pi),\]
where \(\delta_{x_i}\) belongs in case \((B)\), thus \((\ref{u d u d polwsh})\) has to be a strict inequality for all \(x\in[b,\pi)\).
\end{proof}
\begin{proposition}\label{Karamataphi}
Let \(\mu\) be a finite Borel measure without a singular continuous part. Let \(b\in(-\pi,0)\cup(0,\pi)\), let \(I\) be as defined in Lemma \(\ref{karamata green}\)., let \(\mu_H\) be the polarization of \(\mu\) with respect to \(b\). Let \(\phi:\mathbb{R}\to\mathbb{R}\) be a convex and increasing function.
\begin{enumerate}
    \item[(a)] If \(x\in I,\) let \(x'=2b-x\). Then 
    \begin{equation}\label{karamata_phi}
        \phi\bigl(u_\mu(x)\bigr)+\phi\bigl(u_\mu(x')\bigr)\leq\phi\bigl(u_{\mu_H}(x)\bigr)+\phi\bigl(u_{\mu_H}(x')\bigr).
    \end{equation}
    \item[(b)] If \(b>0\) (respectively \(b<0\)), let \(x\in[-\pi,b]\) (respectively \(x\in[b,\pi]\)). Then
    \begin{equation}\label{polar_phi}
        \phi\bigl(u_\mu(x)\bigr)\leq\phi\bigl(u_{\mu_H}(x)\bigr).
    \end{equation}
    If \(\phi\) is strictly increasing and \(b>0\) (respectively \(b<0\)), then the following are equivalent
        \begin{enumerate}
        \item[(i)] \((\ref{polar_phi})\) holds as an equality for some \(x\in(-\pi,b]\) (respectively \(x\in[b,\pi)\)).
        \item[(ii)] \((\ref{polar_phi})\) holds as an equality for all \(x\in(-\pi,b]\) (resp. \(x\in[b,\pi)\)).
        \item[(iii)] \(\mu=\mu_H\).
        \end{enumerate}
    \end{enumerate}
\end{proposition}
\begin{proof}
By Theorem \ref{theorem max}, the ranges of functions \(u_\mu\), \(u_{\mu_H}\) and \(u_{\mu^\#}\) are all subintervals of \([0,u_{\mu^\#}(0)]\). Trivially, the class of functions that are convex and increasing on \([0,u_{\mu^\#}(0)]\) contains the functions that are convex and increasing on the whole real line. The proof works for both cases.\\
(a) Combining Lemma \ref{lemma karamata u f} for the absolutely continuous part of \(\mu\) and Lemma \ref{karamata u d} for the purely discontinuous part of \(\mu\), we get that
\[u_{\mu_H}(x')\geq\max\{u_\mu(x),u_\mu(x')\}\]
and 
\[u_\mu(x)+u_\mu(x')\leq u_{\mu_H}(x)+u_{\mu_H}(x').\]
We apply Lemma \ref{karamata} for \(M=u_{\mu^\#}(0)\),
\[\{x_1,x_2\}=\{u_\mu(x),u_\mu(x')\} \text{ and } \{y_1,y_2\}=\{u_{\mu_H}(x),u_{\mu_H}(x')\}\]
and we get the desired inequality for all functions \(\phi\) that are convex and increasing on \([0,u_{\mu^\#}(0)]\) and thus for all real functions that are convex and increasing on the whole real line.\\
(b) The inequality \((\ref{polar_phi})\) is a direct consequence of Lemma \ref{lemma karamata u f}(c), Lemma \ref{karamata u d}(c) and the monotonicity of \(\phi\). Assuming that \(\phi\) is strictly increasing, we have equality here if and only if equality holds on Lemma \ref{lemma karamata u f}(c) and on Lemma \ref{karamata u d}(c), i.e. if \(f=f_H\) a.e. on \([-\pi,\pi]\) and \(\delta=\delta_H\).
\end{proof}
\section{Proofs of main theorems}\(\phantom{.}\)\\
\textbf{Proof of Theorem 3.8.} Without loss of generality, assume that \(b\in(0,\pi).\)\\
We integrate both sides of inequality \((\ref{karamata_phi})\) from \(b\) to \(\pi\) and after a change of variables, we get
\begin{equation}\label{proof means u f H 1}
    \int_{2b-\pi}^\pi \phi\bigl(u_\mu(x)\bigr)dx\leq\int_{2b-\pi}^\pi \phi\bigl(u_{\mu_H}(x)\bigr)dx.
\end{equation}
Similarly, by \((\ref{polar_phi})\),
\begin{equation}\label{proof means u f H 2}
    \int_{-\pi}^{2b-\pi} \phi\bigl(u_\mu(x)\bigr)dx\leq\int_{-\pi}^{2b-\pi} \phi\bigl(u_{\mu_H}(x)\bigr)dx.
\end{equation}
The inequalities \((\ref{proof means u f H 1})\) and \((\ref{proof means u f H 2})\) combined give the inequality \((\ref{ineq polar})\).
\par Assume that \(\phi\) is strictly increasing. In order for \((\ref{ineq polar})\) to hold as an equality, \((\ref{proof means u f H 2})\) must hold as an equality. If \(\mu=\mu_H\), then obviously \((\ref{proof means u f H 1})\), \((\ref{proof means u f H 2})\) and \((\ref{ineq polar})\) hold as equalities.
\par Conversely, assume that the measures \(\mu\) and \(\mu_H\) are not the same. Then, by Proposition \(\ref{Karamataphi}\)(b),
\[\phi\bigl(u_\mu(x)\bigr)\leq\phi\bigl(u_{\mu_H}(x)\bigr) \text{ for all } x\in[-\pi,b]\supset[-\pi,2b-\pi],\] 
and thus \((\ref{proof means u f H 2})\) holds as a strict inequality, therefore \((\ref{ineq polar})\) holds as a strict inequality as well.\qed
\begin{remark}\label{remark Lp norms polar}
\textnormal{As we saw in the proof of Proposition 4.5 we could have assumed that \(\phi\) is convex and increasing on \([0,M]\), where \(M=u_{\mu^\#}(0)\). Similarly, Theorem 3.8 and its proof work for such functions as well. Consider the convex and strictly increasing function \(\phi(\cdot)=|\cdot|^p:[0,M]\to\mathbb{R}\), where \(p\geq1\). Because \(u_\mu\) and \(u_{\mu_H}\) are non-negative, Theorem 3.8 implies that
\begin{equation}\label{Lp polar}
||u_\mu||_{L^p}\leq ||u_{\mu_H}||_{L^p} \text{ for all }p\in[1,+\infty).
\end{equation}
Letting \(p\to+\infty\) we get the inequality for their \(L^\infty\)-norms:
\[||u_\mu||_\infty\leq||u_{\mu_H}||_\infty.\]}
\end{remark}
The following lemma uses a special case of Theorem 6.1 of \cite{Brock}.
\begin{lemma}\label{seq of polar to sdr}
Let \(0\leq f\in L^1[-\pi,\pi]\). There exists a sequence of non-zero numbers \(\{b_n\}_{n=1}^\infty\subset(-\pi,\pi)\) such that
\[f_n\longrightarrow f^\#\quad\textit{ in }\:\:L^1[-\pi,\pi],\]
where \(f_1\) is the polarization of \(f\) with respect to \(b_1\) and \(f_{n+1}\) is the polarization of \(f_n\) with respect to \(b_{n+1}\), for \(n=1,2,\dots.\)
\par Let \(u_{f_{n}}\) and \(u_{f^\#}\) be the solutions to the corresponding Dirichlet problems. Then, \(u_{f_{n}}\longrightarrow u_{f^\#}\) uniformly on \([-\pi,\pi]\).
\end{lemma}
\begin{proof}
In our context, the existence of a sequence \(\{f_n\}_{n=1}^\infty\) of iterated polarizations, as described in the statement of Lemma \ref{seq of polar to sdr}, such that \(f_n\longrightarrow f^\#\) in \(L^1[-\pi,\pi]\), is directly guaranteed by Theorem 6.1 of \cite{Brock}.
\par Finally, for all \(x\in[-\pi,\pi]\), we have
\begin{align*}
\bigl|u_{f_{n}}(x)-u_{f^\#}(x)\bigr|&=\left|\int_{-\pi}^\pi G(x,y)\bigl(f_n(y)-f^\#(y)\bigr)dy\right|\\
&\leq \max_{x,y\in[-\pi,\pi]}G(x,y)\int_{-\pi}^\pi |f_n(y)-f^\#(y)|dy\\
&=G(0,0)||f_n-f^\#||_{L^1}\xrightarrow{n\to\infty}0. \qedhere
\end{align*}
\end{proof}
We are going to use what in the literature is known as Baernstein's star function. More about star functions can be found in \cite{Baernstein}, specifically in Chapters 9 and 10. In our paper, we only need the star functions for the Lebesgue measure.
\begin{definition}\textnormal{
Let \(f\in L^1[-\pi,\pi]\). We define \(f^\star:[0,2\pi]\to\mathbb{R}\) by
\[f^\star(t)=\sup\int_E f d\lambda,\]
where the supremum is taken over all Borel measurable subsets \(E\) of \([-\pi,\pi]\) with \(\lambda(E)=t\).}
\end{definition}
The following propositions present all the properties of the \(\star\)-function that are relative to our goals. See Propositions 9.2 and Proposition 10.1 of \cite{Baernstein}, respectively.
\begin{proposition}\label{properties of star}
Let \(f\in L^1[-\pi,\pi]\). Then:
\begin{itemize}
    \item[(a)] For each \(t\in[0,2\pi]\) there exists a Borel set \(E\subset[-\pi,\pi]\) with \(\lambda(E)=t\) such that 
    \[f^\star(t)=\int_E fd\lambda.\]
    \item[(b)] For all \(t\in[0,2\pi]\) \(f^\star(t)=\displaystyle\int_0^tf^\ast(s)ds\), where \(f^*\) is the decreasing rearrangement of \(f\).
\end{itemize}
In particular, \((b)\) implies that
\[f^\star(t)=\int_{-t/2}^{t/2}f^\#(s)ds,\]
where \(f^\#\) is the symmetric decreasing rearrangement of \(f.\)
\end{proposition}
\begin{proposition}\label{prop star means}
Let \(f,g\in L^1[-\pi.\pi]\). The following are equivalent.
\begin{itemize}
    \item[(a)] \(\displaystyle\int_{-\pi}^\pi\phi\bigl(f(x)\bigr)dx\leq\int_{-\pi}^\pi\phi\bigl(g(x)\bigr)dx\) for all convex and increasing functions \(\phi:\mathbb{R}\to\mathbb{R}.\)
    \item[(b)]  For any fixed \(y\in\mathbb{R}\), \(\displaystyle\int_{-\pi}^\pi\bigl(f(x)-y\bigr)^+dx\leq\int_{-\pi}^\pi\bigl(g(x)-y\bigr)^+dx\), where \((\cdot)^+=\max\{\cdot,0\}.\)
    \item[(c)] \(f^\star(t)\leq g^\star(t)\) for all \(t\in[0,2\pi].\)
\end{itemize}
\end{proposition}
\begin{lemma}\label{lemma star} Let \(\sigma\) be a finite Borel measure on \((-\pi,\pi)\) such that \(\sigma\perp\lambda\). Let \(\sigma^\#\) be the symmetrization of \(\sigma\). Then
\begin{equation}\label{star u sigma}
u^\star_\sigma\leq u^\star_{\sigma^\#} \text{ on }[0,2\pi].
\end{equation}
\end{lemma}
\begin{proof}
First, we apply Theorem \(\ref{theorem max}\) for \(\mu\) such that \(d\mu=\chi_Ed\lambda\) and we get the inequality
\begin{equation}\label{lemma green E}
\int_{-\lambda(E)/2}^{\lambda(E)/2} G(0,y)dy\geq \int_E G(x,y)dy \text{ for all }x\in[-\pi,\pi],
\end{equation}
where \(E\) is any measurable subset of \((-\pi,\pi)\).
\par We use this inequality and apply Fubini's Theorem twice to get
\begin{align*}
u_\sigma^\star(t)&=\sup_{\lambda(E)=t}\int_E u_\sigma(s)ds=\sup_{\lambda(E)=t}\int_E\int_{-\pi}^\pi G(s,y)d\sigma(y)ds\\
&=\sup_{\lambda(E)=t}\int_{-\pi}^\pi\int_E G(s,y)dsd\sigma(y)\leq\int_{-\pi}^\pi\int_{-t/2}^{t/2} G(0,s)dsd\sigma(y)\\
&=\int_{-t/2}^{t/2}\int_{-\pi}^\pi G(0,s)d\sigma(y)ds=\int_{-t/2}^{t/2}\sigma\bigl((-\pi,\pi)\bigr)G(0,s)ds.
\end{align*}
Note that \((u_{\sigma^\#})^\#(s)=u_{\sigma^\#}(s)=\sigma\bigl((-\pi,\pi)\bigr)G(0,s)\), thus, by Proposition \ref{properties of star}, we have
\[u_{\sigma^\#}^\star(t)=\int_{-t/2}^{t/2}(u_{\sigma^\#})^\#(s)ds=\int_{-t/2}^{t/2}\sigma\bigl((-\pi,\pi)\bigr)G(0,s)ds.\]
This concludes our proof.
\end{proof}
\begin{theorem}\label{means sdr f}
Let \(0\leq f\in L^1[-\pi,\pi]\). Then
\begin{equation}\label{ineq sdr f}
\int_{-\pi}^\pi \phi\bigl(u_f(x)\bigr)dx\leq\int_{-\pi}^\pi \phi\bigl(u_{f^\#}(x)\bigr)dx
\end{equation}
for any convex and increasing function \(\phi:\mathbb{R}\to\mathbb{R},\) and, thus
\begin{equation}\label{star f}
u_f^\star(t)\leq u^\star_{f^\#}(t) \text{ for all } t\in[0,2\pi].
\end{equation}
Furthermore, if \(\phi\) is strictly increasing, then \((\ref{ineq sdr f})\) holds as an equality if and only if \(f=f^\#\) a.e. on \((-\pi,\pi)\).
\end{theorem}
Note that inequality \((\ref{means sdr f})\) has been proven in \cite[Corollary 3.1]{Langford}. To achieve this, the authors viewed the solution to the Dirichlet problem as the limit of solutions to Robin problems. We will use a different method to prove our theorem, namely repeated polarizations. This will allow us to examine the equality case.
\begin{proof}
Let \(\{f_n\}_{n=1}^\infty\) be as in Lemma \(\ref{seq of polar to sdr}\). We apply Theorem \(\ref{polarconvex}\) for the measures such that \(d\mu_n=f_nd\lambda\) and get
\begin{align*}
    \int_{-\pi}^\pi\phi\bigl(u_f(x)\bigr)dx&\leq \int_{-\pi}^\pi\phi\bigl(u_{f_1}(x)\bigr)dx\leq \int_{-\pi}^\pi\phi\bigl(u_{f_2}(x)\bigr)dx\leq\dots\\
    &\leq \int_{-\pi}^\pi\phi\bigl(u_{f_n}(x)\bigr)dx\leq\dots.
\end{align*}
So the sequence \(\displaystyle\int_{-\pi}^\pi\phi\bigl(u_{f_n}(x)\bigr)dx\) is increasing in \(n\). By the uniform convergences of \(\{u_{f_n}\}_{n=1}^\infty\) to \(u_{f^\#}\), which is assured by Lemma \(\ref{seq of polar to sdr}\), we have
\[\lim_{n\to\infty}\int_{-\pi}^\pi\phi\bigl(u_{f_n}(x)\bigr)dx=\int_{-\pi}^\pi\phi\bigl(u_{f^\#}(x)\bigr)dx\]
which implies that
\begin{equation}
    \int_{-\pi}^\pi\phi\bigl(u_f(x)\bigr)dx\leq\int_{-\pi}^\pi\phi\bigl(u_{f^\#}(x)\bigr)dx.
\end{equation}
By Proposition \(\ref{prop star means}\), the inequality \((\ref{ineq sdr f})\) holds for every convex and increasing \(\phi:\mathbb{R}\to\mathbb{R}\) if and only if \[u_f^\star\leq u^\star_{f^\#} \text{ on } [0,2\pi].\]
If \(f=f^\#\) a.e. on \((-\pi,\pi)\), then \((\ref{ineq sdr f})\) holds as an equality, obviously.
\par Conversely, assume that \(\phi\) is strictly increasing and that \(f\) is not a.e. equal to its s.d.r.. Then, by \cite[Lemma 1.38]{Baernstein}, there exists some polarization \(f_H\) that is not a.e. equal to \(f\). By Theorem \(\ref{polarconvex}\), this means that
\[\int_{-\pi}^\pi\phi\bigl(u_f(x)\bigr)dx<\int_{-\pi}^\pi\phi\bigl(u_{f_H}(x)\bigr)dx.\]
We have already proven the first part of Theorem \(\ref{means sdr f}\), so we can apply it to \(f_H\), and get
\[\int_{-\pi}^\pi\phi\bigl(u_{f_H}(x)\bigr)dx\leq\int_{-\pi}^\pi\phi\bigl(u_{{(f_H)}^\#}(x)\bigr)dx.\]
But, note that \((f_H)^\#=f^\#\) since \(f,\) \(f_H\) and \(f^\#\) are all rearrangements of each other, and thus
\[\int_{-\pi}^\pi\phi\bigl(u_f(x)\bigr)dx<\int_{-\pi}^\pi\phi\bigl(u_{f_H}(x)\bigr)dx\leq\int_{-\pi}^\pi\phi\bigl(u_{f^\#}(x)\bigr)dx.\]
\end{proof}
\begin{lemma}\label{lemma star mu}
Let \(\mu\) be a finite Borel measure on \((-\pi,\pi)\). Then
\begin{equation}\label{star mu}
u_{\mu}^\star\leq u_{\mu^\#}^\star \text{ on }[0,2\pi].
\end{equation}
\end{lemma}
\begin{proof}
Let \(0\leq f\in L^1[-\pi,\pi]\) and let \(\sigma\) be the singular measure on \((-\pi,\pi)\) such that \(d\mu=fd\lambda+d\sigma\). Note that \(u_\mu(x)=u_f(x)+u_\sigma(x)\) and \(u_{\mu^\#}(x)=u_{f^\#}(x)+u_{\sigma^\#}(x)\) for all \(x\in[-\pi,\pi]\). Note also that \(u_f\) is a concave function; (the proof of this is an easy Calculus exercise  using the representation of \(u_f\) given by Proposition \(\ref{Representation}\)). Even more directly, one can see that \(u_{f^\#}\) is an even function and, because it is also concave, we have \((u_{f^\#})^\#=u_{f^\#}\). Furthermore, \(u_{\sigma^\#}\) is also its own s.d.r., quite trivially (see Example \(\ref{example dirac}\)), and, thus, \(u_{\mu^\#}\) is equal to its own s.d.r. as well.
\par We can now prove our desired inequality using all these preliminary facts, Proposition \(4.7\), Lemma \(\ref{lemma star}\) and Theorem \(\ref{means sdr f}\) as follows:
\begin{align*}
u_\mu^\star(t)&=\sup_{\lambda(E)=t}\int_E u_\mu(s)ds=\sup_{\lambda(E)=t}\int_E\bigl(u_f(s)+u_\sigma(s)\bigr)ds\\
&\leq\sup_{\lambda(E)=t}\int_Eu_f(s)ds+\sup_{\lambda(E)=t}\int_E u_\sigma(s)ds=u_f^\star(t)+u_\sigma^\star(t)\\
&\leq u^\star_{f^\#}(t)+u^\star_{\sigma^\#}(t)=\int_{-t/2}^{t/2}\bigl(u_{f^\#}(s)\bigr)^\#ds+\int_{-t/2}^{t/2}\bigl(u_{\sigma^\#}(s)\bigr)^\#ds\\
&=\int_{-t/2}^{t/2}u_{f^\#}(s)ds+\int_{-t/2}^{t/2}u_{\sigma^\#}(s)ds=\int_{-t/2}^{t/2}\bigl(u_{f^\#}(s)+u_{\sigma^\#}(s)\bigr)ds\\
&=\int_{-t/2}^{t/2}u_{\mu^\#}(s)ds=\int_{-t/2}^{t/2}\bigl(u_{\mu^\#}(s)\bigr)^\#ds=u_{\mu^\#}(t).
\end{align*}
\end{proof}
\justify
\textbf{Proof of Theorem 1.3.}  Let \(d\mu=fd\lambda+d\sigma+d\delta\), where \(0\leq f\in L^1[-\pi,\pi]\), \(\sigma\) is a continuous measure on \((-\pi,\pi)\), \(\sigma\) and \(\delta\) are singular with respect to \(\lambda\) and \(\delta=\sum_i a_i \delta_{x_i}\) is as described in Definition \ref{decompdef}.
Inequality \((\ref{ineq sdr})\) is a direct consequence of inequality \((\ref{star mu})\) and Proposition \(\ref{prop star means}\).\\
Now, assume that \(\phi\) is strictly convex. Then \(\phi\) is also strictly increasing, so if \(\mu\) is absolutely continuous, the uniqueness statement of this theorem is a consequence of Theorem \(\ref{means sdr f}\).
\par For the general case, inspired by the proof of \cite[Proposition 3]{Baernstein1974}, we can decompose \(\phi\) in the following way:
\[\phi=\phi_1+\phi_2,\]
where \(\phi_1,\phi_2:\mathbb{R}\to\mathbb{R}\) are convex and increasing such that \(\phi_1=\phi\) on \((-\infty,0]\), \(\phi_1\) is linear on \([0,+\infty)\) with slope \(\alpha=\displaystyle\lim_{t\to0^-}\phi'(t)\geq0\), \(\phi_2=\phi(0)\) on \((-\infty,0]\) and \(\phi_2=\phi-\phi_1\) on \([0,+\infty)\). Then \(\phi_2\) is constant on \((-\infty,0]\) and strictly convex on \([0,+\infty).\) Additionally, we have that \(\phi_2\) is differentiable \(\lambda\)-almost everywhere and \(\phi_2'\), where it exists, is increasing. Moreover, \(\displaystyle\lim_{s\to{s_0}^-}\phi_2'(s)\) exists for every \(s_0\in \mathbb{R}\). Since \(\phi_2'\) is increasing and left-continuous and \(\displaystyle\lim_{s\to-\infty}\phi_2'(s)=0\), there exists a positive Borel measure \(\nu\) on \(\mathbb{R}\) such that \(\displaystyle\lim_{s\to{s_0}^-}\phi_2'(s)=\nu\bigl((-\infty,s_0)\bigr)\) for every \(s_0\in\mathbb{R}\). Note that \(\displaystyle\lim_{x\to{s_0}^-}\phi_2'(s)=0\) for every \(s_0\leq0\). From now on, let \(\phi_2'\) denote the left side-limit of the derivative of \(\phi_2\). For \(s_0>0\) we have
\begin{align}\label{representation phi 2}
\phi_2(s_0)-\phi(0)&=\int_{-\infty}^{s_0} \phi_2'(t)dt=-\int_{-\infty}^{s_0} \phi_2'(t)\frac{d}{dt}(s_0-t)dt\\
\nonumber&=\int_{(-\infty,s_0)}(s_0-t)d\nu(t)\\
\nonumber&=\int_\mathbb{R}(s_0-t)^+d\nu(t).
\end{align}
The same formula for \(\phi_2\) holds trivially for \(s_0\leq0\) since in that case, the integral on the right hand side is identically zero. More details about the measure \(\nu\) and the integration by parts that we used can be found on Chapter 3 of \cite{Folland}. Another fact we need to point out about \(\nu\) is that, since \(\phi_2\) is strictly convex on \((0,+\infty)\), we have \(\nu(J)>0\) for every interval \(J\) that contains some non-empty subset of \((0,+\infty)\).
\par Assume that \(\mu\neq\mu^\#\). By Theorem \ref{theorem max}, we have \(\displaystyle\max_{[-\pi,\pi]}u_\mu<u_{\mu^\#}(0)\). Consider the interval \(J=\bigl(\max u_\mu,u_{\mu^\#}(0)\bigr)\). Since \(0<\max u_\mu\), we have \(\nu(J)>0\). We apply the first part of Theorem \ref{sdrconvex} to the convex and increasing \(\phi_2\) and we get 
\begin{equation}\label{convex phi2}
\int_{-\pi}^\pi \phi_2\bigl(u_\mu(x)\bigr)dx\leq\int_{-\pi}^\pi \phi_2\bigl(u_{\mu^\#}(x)\bigr)dx.
\end{equation}
Similarly, for every \(t\in\mathbb{R}\), \(\Phi_t(s)=(s-t)^+:\mathbb{R}\to\mathbb{R}\) is non-negative convex and increasing, so, by the first part of Theorem \ref{sdrconvex}, we have
\begin{equation}\label{convex positive part}
0\leq\int_{-\pi}^\pi\bigl(u_\mu(x)-t\bigr)^+dx\leq\int_{-\pi}^\pi\bigl(u_{\mu^\#}(x)-t\bigr)^+dx.
\end{equation}
By \((\ref{representation phi 2})\) and \((\ref{convex phi2})\) and the representation we derived for \(\phi_2\), we get the inequality
\begin{equation}
\int_{-\pi}^\pi\int_\mathbb{R}\bigl(u_\mu(x)-t\bigr)^+d\nu(t)dx\leq\int_{-\pi}^\pi\int_\mathbb{R}\bigl(u_{\mu^\#}(x)-t\bigr)^+d\nu(t)dx.
\end{equation}
For \(t\in J\), \(u_\mu(x)<t\) and therefore
\[\int_{-\pi}^\pi\int_J \bigl(u_\mu(x)-t\bigr)^+dv(t)dx=0.\]
So, we get the following two inequalities:
\begin{equation*}
\int_{-\pi}^\pi\int_{\mathbb{R}\setminus J}\bigl(u_\mu(x)-t\bigr)^+d\nu(t)dx\leq\int_{-\pi}^\pi\int_{\mathbb{R}\setminus J}\bigl(u_{\mu^\#}(x)-t\bigr)^+d\nu(t)dx
\end{equation*}
and
\begin{equation*}
0=\int_{-\pi}^\pi\int_J\bigl(u_\mu(x)-t\bigr)^+d\nu(t)dx\leq\int_{-\pi}^\pi\int_J\bigl(u_{\mu^\#}(x)-t\bigr)^+d\nu(t)dx.
\end{equation*}
With the application of Fubini's Theorem, these inequalities give us
\begin{equation}\label{ineq eksw apo to J}
\int_{\mathbb{R}\setminus J}\int_{-\pi}^\pi\bigl(u_\mu(x)-t\bigr)^+dxd\nu(t)\leq\int_{\mathbb{R}\setminus J}\int_{-\pi}^\pi\bigl(u_{\mu^\#}(x)-t\bigr)^+dxd\nu(t)
\end{equation}
and
\begin{equation*}
0=\int_J\int_{-\pi}^\pi\bigl(u_\mu(x)-t\bigr)^+dxd\nu(t)\leq\int_J\int_{-\pi}^\pi\bigl(u_{\mu^\#}(x)-t\bigr)^+dxd\nu(t),
\end{equation*}
respectively.
\par Note that \(u_{\mu^\#}(0)>0\) and for \(t\in J\), we have \(\bigl(u_\mu(x)-t\bigr)^+=0\) for all \(x\in[-\pi,\pi]\). Since \(u_{\mu^\#}\) is continuous and \(u_{\mu^\#}(0)>t\) for \(t\in J\), there exists \(\varepsilon>0\) such that \(\bigl(u_{\mu^\#}(x)-t\bigr)^+>0\) for all \(x\in(-\varepsilon,\varepsilon)\). Thus, 
\begin{equation*}
0=\int_{-\pi}^\pi\bigl(u_\mu(x)-t\bigr)^+dx<\int_{-\varepsilon}^\varepsilon\bigl(u_{\mu^\#}(x)-t\bigr)^+dx\leq\int_{-\pi}^\pi\bigl(u_{\mu^\#}(x)-t\bigr)^+dx
\end{equation*}
for all \(t\in J\), and, since \(\nu(J)>0\),
\begin{equation}\label{ineq J}
0=\int_J\int_{-\pi}^\pi\bigl(u_\mu(x)-t\bigr)^+dxd\nu(t)<\int_J\int_{-\pi}^\pi\bigl(u_{\mu^\#}(x)-t\bigr)^+dxd\nu(t).
\end{equation}
Combining inequalities \((\ref{ineq eksw apo to J})\) and \((\ref{ineq J})\) we get
\begin{equation*}
\int_{\mathbb{R}}\int_{-\pi}^\pi\bigl(u_\mu(x)-t\bigr)^+dxd\nu(t)<\int_{\mathbb{R}}\int_{-\pi}^\pi\bigl(u_{\mu^\#}(x)-t\bigr)^+dxd\nu(t).
\end{equation*}
Applying Fubini's Theorem once again, we get
\begin{equation*}
\int_{-\pi}^\pi\int_{\mathbb{R}}\bigl(u_\mu(x)-t\bigr)^+d\nu(t)dx<\int_{-\pi}^\pi\int_{\mathbb{R}}\bigl(u_{\mu^\#}(x)-t\bigr)^+d\nu(t)dx,
\end{equation*}
which in turn implies that
\begin{equation}\label{convex strict phi2}
\int_{-\pi}^\pi \phi_2\bigl(u_\mu(x)\bigr)dx<\int_{-\pi}^\pi \phi_2\bigl(u_{\mu^\#}(x)\bigr)dx.
\end{equation}
Finally, we apply the first part of Theorem \ref{sdrconvex} to \(\phi_1\) and get
\begin{equation}\label{convex phi1}
\int_{-\pi}^\pi \phi_1\bigl(u_\mu(x)\bigr)dx\leq\int_{-\pi}^\pi \phi_1\bigl(u_{\mu^\#}(x)\bigr)dx.
\end{equation}
Combining \((\ref{convex strict phi2})\) and \((\ref{convex phi1})\) and recalling that \(\phi=\phi_1+\phi_2\), we get the desired strict inequality
\begin{equation*}
\int_{-\pi}^\pi \phi\bigl(u_\mu(x)\bigr)dx<\int_{-\pi}^\pi \phi\bigl(u_{\mu^\#}(x)\bigr)dx.
\end{equation*}
\qed
\section{Inequalities for the \texorpdfstring{\(L^p-\)}{TEXT}norms} 
\par Finally, we present a theorem that involves the symmetrization of the measure and the \(L^p-\)norms of solutions.
\begin{theorem}\label{symmetric Lp norms}
Let \(\mu\) be a finite Borel measure on \((-\pi,\pi)\) and let \(\mu^\#\) be its symmetrization. Let \(u_\mu\) and \(u_{\mu^\#}\) be the corresponding solutions to the Dirichlet problem \((\ref{Poisson})\), \((\ref{Dirichlet})\). Then
\begin{equation}\label{inequality of Lp norms of symmetric}
||u_\mu||_{L^p}\leq||u_{\mu^\#}||_{L^p} \text{, for }1\leq p\leq+\infty.
\end{equation}
Moreover, \((\ref{inequality of Lp norms of symmetric})\) holds as an equality if and only if the Radon--Nikodym derivative \(f\) of \(\mu\) is equal to its s.d.r. \(f^\#\) a.e. on \((-\pi,\pi)\) and the singular part of \(\mu\) is a Dirac mass at the origin.
\end{theorem}
\begin{proof}
 Let \(p\geq1\) and consider the function 
\[
\phi(s)=
\begin{cases}
0, & s\leq0,\\
s^p, &s\geq0.
\end{cases}
\]
By Theorem \ref{sdrconvex} and because \(u_\mu\) and \(u_{\mu^\#}\) are non-negative, we deduce that
\[||u_\mu||_p\leq||u_{\mu^\#}||_p\text{ for all }p\in[1,+\infty).\]
If \(p>1\), then \(\phi\) has the same properties as \(\phi_2\) in the proof of Theorem \ref{sdrconvex}, thus, this inequality is strict if \(\mu\neq\mu^\#.\) Theorem \ref{polarconvex} gives us the inequality for the \(L^\infty-\)norms and the uniqueness statement. Finally, we prove uniqueness statement for the case \(p=1\):\\
Let \(\mu=\nu+\sigma+\delta\) with the usual assumptions, then \(u_\mu=u_\nu+u_\sigma+u_\delta\). By Remark \ref{remark Lp norms polar}, Theorem \ref{polarconvex} is valid for \(\phi\), and applying Theorem \ref{sdrconvex}, we have
\begin{align}
\label{L1 norma ineq 1}\int_{-\pi}^\pi u_{\nu+\delta}(x)dx&=||u_{\nu+\delta}||_1\leq||u_{\nu_H+\delta_H}||_1\\
\nonumber&\leq||u_{\nu^\#+\delta^\#}||_1=\int_{-\pi}^\pi u_{\nu^\#+\delta^\#}(x)dx.
\end{align}
Because \(\phi\) is strictly increasing on \([0,u_{\mu^\#}(0)]\), \((\ref{L1 norma ineq 1})\) holds as an equality if and only if \(\nu=\nu_H\) and \(\delta=\delta_H\).\\
Assume that \(\nu+\delta\) is not the zero measure and \(\nu+\delta\neq\nu^\#+\delta^\#\), then \(\mu\neq\mu^\#\) and there exists \(b\in(-\pi,0)\cup(0,\pi)\) such that \(\nu+\delta\neq\nu_H+\delta_H\) and
\[||u_{\nu+\delta}||_1<||u_{\nu_H+\delta_H}||_1\leq||u_{\nu^\#+\delta^\#}||_1.\]
By Theorem \ref{sdrconvex}
\[||u_\sigma||_1\leq||u_{\sigma^\#}||_1,\]
and combining this inequality with the previous one we get (because the solutions are non-negative)
\[||u_\mu||_1<||u_{\mu^\#}||_1.\]
Assume that \(\sigma\) is not the zero measure, then \(\mu\neq\mu^\#\). The symmetrization of \(\sigma\) is the measure \(\sigma^\#=\sigma\bigl((-\pi,\pi)\bigr)\delta_0\), so, by Example \ref{example dirac}, we have
\begin{align}
\nonumber||u_{\sigma^\#}||_1&=\int_{-\pi}^\pi u_{\sigma^\#}(x)dx=\int_{-\pi}^\pi\int_{-\pi}^\pi G(x,y)d\sigma^\#(y)dx\\
\nonumber&=\int_{-\pi}^\pi\sigma\bigl((-\pi,\pi)\bigr)G(x,0)dx=\sigma\bigl((-\pi,\pi)\bigr)\frac{\pi^2}{2}.
\end{align}
To estimate the other \(L^1-\)norm, applying Fubini's theorem, we get
\begin{align}
\nonumber||u_\sigma||_1&=\int_{-\pi}^\pi u_{\sigma}(x)dx=\int_{-\pi}^\pi\int_{-\pi}^\pi G(x,y)d\sigma(y)dx\\
\nonumber&=\int_{-\pi}^\pi\int_{-\pi}^\pi G(x,y)dxd\sigma(y)=\int_{-\pi}^\pi\frac{\pi^2-y^2}{2}d\sigma(y)\\
\label{austhrh singular L apeira}&=||u_{\sigma^\#}||_1-\frac{1}{2}\int_{-\pi}^\pi y^2d\sigma(y)\leq||u_{\sigma^\#}||_1.
\end{align}
Because \(\sigma\) is not a Dirac mass at 0, there exists \(\varepsilon>0\) such that \(\sigma(U)>0\), where \(U=(-\pi,\pi)\setminus(-\varepsilon,\varepsilon).\) Therefore,
\[\int_{-\pi}^\pi y^2d\sigma(y)\geq \int_U y^2d\sigma(y)\geq \int_U \min_{U}y^2d\sigma(y)=\int_U \varepsilon^2d\sigma(y)=\varepsilon^2\sigma(U)>0.\]
Thus, \((\ref{austhrh singular L apeira})\) is a strict inequality.
\end{proof}
\section*{Acknowledgments}
I would like to thank Professor D. Betsakos, my advisor, for his advice during the preparation of this work. I also thank Professor D. Ntalampekos for his helpful suggestions.
\bibliography{bibliography}
\bibliographystyle{plain}
\nocite{*}
\end{document}